\documentclass[review]{elsarticle}

\usepackage{lineno,hyperref,ulem}
\pdfoutput=1
\journal{Journal of \LaTeX\ Templates}
\usepackage{natbib}
\usepackage{subfigure}
\usepackage{amssymb}
\usepackage{dsfont}
\usepackage{setspace}
\usepackage{bbm}
\usepackage{wrapfig}
\usepackage{algorithm}
\usepackage{algpseudocode}
\usepackage{epsfig,amssymb,latexsym,color,amsmath,pifont,colordvi,multicol}

\journal{Journal of Nonlinear Science}
\definecolor{backgrey}{rgb}{0.86,0.86,0.86}
\definecolor{dblue}{rgb}{0,0.0,0.5}
\definecolor{dred}{rgb}{0.4,0.2,0}
\definecolor{dgreen}{rgb}{0.0,0.5,0}

\renewcommand{\baselinestretch}{1}

\newcommand{\captionfonts}{\small}
\makeatletter  
\long\def\@makecaption#1#2{%
  \vskip\abovecaptionskip
  \sbox\@tempboxa{{\captionfonts #1: #2}}%
  \ifdim \wd\@tempboxa >\hsize
    {\captionfonts #1: #2\par}
  \else
    \hbox to\hsize{\hfil\box\@tempboxa\hfil}%
  \fi
  \vskip\belowcaptionskip}
\makeatother   
\newtheorem{theorem}{Theorem}

\newtheorem{remark}[theorem]{Remark}
\newtheorem{properties}[theorem]{Property}

\newtheorem{definition}[theorem]{Definition}

\newtheorem{proposition}[theorem]{Proposition}

\newenvironment{proof}[1][Proof]{\textbf{#1.} }{\ \hspace*{\fill} \rule{0.5em}{0.5em}}

\begin{document}

\begin{frontmatter}

\title{On Few Shot Learning of Dynamical Systems: A Koopman Operator Theoretic Approach}

\author{Suhbrajit Sinha*}\cortext[mycorrespondingauthor]{Corresponding author}
\ead{subhrajit.sinha@pnnl.gov}
%
%
\author{Umesh Vaidya}
\ead{uvaidya@clemson.edu}
\author{Enoch Yeung}
\ead{eyeung@ucsb.edu}

\fntext[fn1]{S. Sinha is with the Earth and Biological Science Directorate, Pacific Northwest National Laboratory
Richland, WA, 99354}
\fntext[fn2]{U. Vaidya is with the Department of Mechanical Engineering,
Clemson University,
Clemson, SC 29634}

\fntext[fn3]{E. Yeung is with the Department of Mechanical Engineering,
University of California Santa Barbara,
Santa Barbara, CA 93106}




\begin{abstract}
In this paper, we propose a novel algorithm for learning the Koopman operator of a dynamical system from a \textit{small} amount of training data. In many applications of data-driven modeling, e.g. biological network modeling, cybersecurity, modeling the Internet of Things, or smart grid monitoring, it is impossible to obtain regularly sampled time-series data with a sufficiently high sampling frequency.   In such situations the existing Dynamic Mode Decomposition (DMD) or Extended Dynamic Mode Decomposition (EDMD) algorithms for Koopman operator computation often leads to a low fidelity approximate Koopman operator. To this end, this paper proposes an algorithm which can compute the Koopman operator efficiently when the training data-set is sparsely sampled across time. In particular, the proposed algorithm enriches the small training data-set by appending artificial data points, which are treated as noisy observations. The larger, albeit noisy data-set is then used to compute the Koopman operator, using techniques from Robust Optimization. The efficacy of the proposed algorithm is also demonstrated on three different dynamical systems, namely a linear network of oscillators, a nonlinear system and a dynamical system governed by a Partial Differential Equation (PDE).
\end{abstract}


\end{frontmatter}


\section{Introduction}
Dynamical systems theory had started with the works of Newton \cite{principia} and since then has developed into a mature branch of mathematics and physics with applications to many different branches of science and engineering. Typically, dynamical systems are studied in two different ways. One way is to use techniques from differential geometry, where the evolution of the state is studied on the configuration manifold and the associated tangent and cotangent bundles \cite{marsden_mechanics_book}. The other way studies the evolution of functions of the state or measures of the state flow on the configuration manifold \cite{Lasota}. In particular, the evolution of functions or measures is governed by linear operators on appropriate spaces. Though this exposition leads to an infinite-dimensional operator, a big advantage is the fact that even if the underlying system is nonlinear, in the infinite-dimensional space, the evolution is linear \cite{Lasota}.

In recent years, with the advancements in computational capacity and availability of data, there has been a big drive towards data-driven analysis of systems. In particular, increase in memory, processing powers of computers and the advancement in distributed computing architectures have enabled us to handle and analyze data with increasing precision and address learning problems at an unprecedented scale. On the other hand dynamical systems theory finds application in many different disciplines like complex networks, power networks, biological systems, finance etc. and the advantage of data-driven analysis of dynamical systems is the fact that for many naturally occurring complex systems and engineered systems with emergent phenomena, e.g., biological systems, inter-dependent critical infrastructure, social networks, financial systems, it may not always be possible to derive and analyze theoretical mathematical models of the underlying systems \cite{yeung2015global}. In such cases, one has to resort to data-driven techniques for understanding the behavior of such systems.

Motivated by these applications of data-driven modeling, there has been increasing interest in transfer operator theoretic techniques, namely Perron-Frobenius and Koopman operator techniques, for analysis and control of dynamical systems \cite{mezic2005spectral,Dellnitz_Junge,Mezic2000,froyland_extracting,Junge_Osinga,Mezic_comparison,
Dellnitztransport,Vaidya_TAC, mezic_koopmanism, EDMD_williams,
mezic_koopman_stability, yeung2018koopman,yeung2017learning,sinha_sparse_koopman_acc,sinha_equivariant_IFAC}. In the application front, \cite{surana_observer} used Koopman operators for design of observers for general nonlinear systems. Again, in \cite{optimal_placement_ECC,optimal_placement_JMAA,sinha_IT_optimal_placement_ICC, sinha_IT_optimal_placement_arxiv} the Perron-Frobenius operator was used for control of non-equilibrium dynamics. The Koopman operator, which is adjoint of the Perron-Frobeius operator, has also found applications in many different branches like power networks \cite{susuki2013nonlinear, sinha_computationally_efficient,sinha_online_PES}, identification of causal structure using information transfer and its applications in power systems \cite{sinha_IT_CDC_2015,sinha_IT_CDC_2016,sinha_IT_data_acc,sinha_IT_data_journal,sinha_IT_ICC,sinha_IT_power_CDC,sinha_IT_power_journal}, biological systems \cite{sinha_genetic,hasnain2019optimal} etc.

The major advantage of the operator theoretic framework for analysis and control of dynamical systems is that these methods facilitate data-driven learning of dynamical systems. In particular, a finite-dimensional approximation of both Perron-Frobenius and Koopman operators can be constructed from time-series data obtained from experiments and different data-driven methods for constructing finite-dimensional approximations of these operators have been proposed \cite{dellnitz2002set, Mezic2000,DMD_schmitt,rowley2009spectral,EDMD_williams}. Among these algorithms Dynamic Mode Decomposition (DMD) and Extended Dynamic Mode Decomposition (EDMD) are used most extensively for the computation of the finite-dimensional approximation of the Koopman operators. Recent works have also generalized the algorithms for computation of the transfer operators to account for process and observation noise and for Random Dynamical Systems (RDS)  \cite{mezic_stochastic_koopman_spectrum,PhysRevE.96.033310, robust_DMD_ACC,robust_DMD_journal}. In \cite{mezic_stochastic_koopman_spectrum} the authors have provided a characterization of the spectrum and eigenfunctions of the Koopman operator for discrete and continuous time RDS, while in \cite{PhysRevE.96.033310}, the authors have provided an algorithm to compute the Koopman operator for systems with both process and observation noise. In \cite{robust_DMD_ACC,robust_DMD_journal} the authors used robust optimization-based techniques to compute the approximate Koopman operator for data sets of finite length and have shown that normal DMD or EDMD and subspace DMD \cite{PhysRevE.96.033310} lead to an unsatisfactory approximation of Koopman operator for data sets of finite length.

A different and often practical challenge that researchers have to account for is the scenario when the obtained data set has only few time-points (sparse data) which makes the problem of computation of the Koopman operator ill-posed. Sparse data refers to data sets with few data points and can have an immense effect on the ability to train the Koopman operator into producing accurate predictions. In particular, existing DMD and EDMD algorithms may lead to an ill-conditioned least-square problem. In this paper, we address this specific problem of computation of Koopman operator when the data set has few data points. We append artificial data points to the sparse data set to enrich the data and use robust optimization-based techniques to obtain the approximate Koopman operator. The robust optimization problem is a min-max problem which can be approximated as a least squares problem with a regularization term. The regularization parameter imposes sparsity in the Koopman operator. Moreover, it prevents over-fitting of the data and hence can be used to design a data-driven predictor \cite{korda_mezic_predictor}. Furthermore, we discuss the complexity and performance of the proposed Sparse Koopman Algorithm using the concepts of Vapnik-Chervonenkis (VC) dimension \cite{vapnik_book,mostafa_learning_book} and Bias-Variance Trade-off \cite{mostafa_learning_book} and show that the proposed idea of appending artificial data points do improve the Koopman learning problem. We also demonstrate the efficacy of our algorithm on three different dynamical systems, namely a linear network of oscillators, a nonlinear system and a system governed by a Partial Differential Equation (PDE). 

The organization of the paper is as follows. In section \ref{section_transfer_operators} we provide the basics of transfer operators followed by a discussion of DMD and EDMD algorithms in section \ref{section_DMD}. In section \ref{section_sparse_Koopman} we present the main results of the paper and state the algorithm to construct the Koopman operator for sparse data. Analysis of the performance of the proposed algorithm, based on Vapnik-Chervonenkis dimension is discussed in section \ref{section_VC_analysis} and design of the robust predictor is presented in section \ref{section_predictor}, with simulation results in section \ref{section_simulation}. Finally we conclude the paper in section \ref{section_conclusion}.

\section{Transfer Operators for Dynamical Systems}\label{section_transfer_operators}
Consider a discrete-time dynamical system
\begin{eqnarray}\label{system}
z_{t+1} = T(z_t)
\end{eqnarray}
where $T:Z\subset \mathbb{R}^N \to Z$ is assumed to be an invertible smooth diffeomorphism.  Associated with the dynamical system (\ref{system}) is the Borel-$\sigma$ algebra ${\cal B}(Z)$ on $Z$ and the vector space ${\cal M}(Z)$ of bounded complex valued measures on $X$. With this, two linear operators, namely, Perron-Frobenius (P-F) and Koopman operator, can be defined as follows \cite{Lasota} :
\begin{definition}[Perron-Frobenius Operator] 
The Perrorn-Frobenius operator $\mathbb{P}:{\cal M}(Z)\to {\cal M}(Z)$ is given by
\[[\mathbb{P}\mu](A)=\int_{{\cal Z} }\delta_{T(z)}(A)d\mu(z)=\mu(T^{-1}(A))\]
$\delta_{T(z)}(A)$ is stochastic transition function which measure the probability that point $z$ will reach the set $A$ in one time step under the system mapping $T$. 
\end{definition}

\begin{definition}[Invariant measures] Invariant measures are the fixed points of
the P-F operator $\mathbb{P}$ that are also probability measures. Let $\bar \mu$ be the invariant measure then, $\bar \mu$ satisfies
\[\mathbb{P}\bar \mu=\bar \mu.\]
\end{definition}

If the state space $Z$ is compact, it is known that the P-F operator admits at least one invariant measure.

\begin{definition} [Koopman Operator] 
Given any $h\in\cal{F}$, $\mathbb{U}:{\cal F}\to {\cal F}$ is defined by
\[[\mathbb{U} h](z)=h(T(z))\]
where $\cal F$ is the space of function (observables) invariant under the action of the Koopman operator.
\end{definition}

\begin{figure}[htp!]
\centering
\includegraphics[scale=.4]{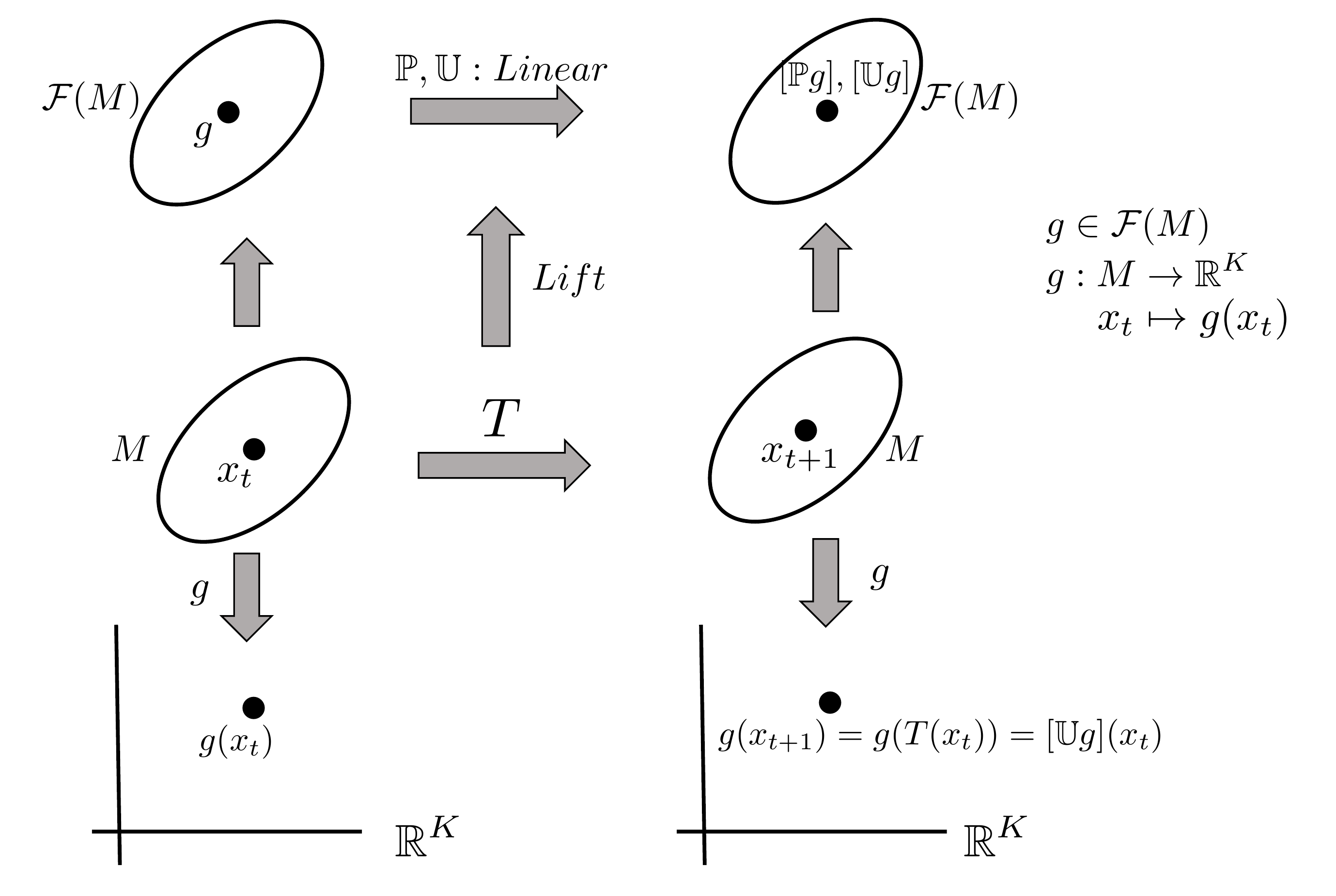}
\caption{Perron-Frobenius and Koopman operators corresponding to a dynamical system.}\label{koopman_diagram}
\end{figure}

Both the P-F operator and the Koopman operator are linear operators, even if the underlying system is non-linear. But while analysis is made tractable by linearity, the trade-off is that these operators are typically infinite dimensional. In particular, the P-F operator and Koopman operator often will lift a dynamical system from a finite-dimensional space to generate an infinite dimensional linear system in infinite dimensions. 
\begin{properties}\label{property}
Following properties for the Koopman and Perron-Frobenius operators can be stated \cite{Lasota}.

\begin{enumerate}
\item [a).] For the Hilbert space ${\cal F}=L_2(Z,{\cal B}, \bar \mu)$ 
\begin{eqnarray*}
&&\parallel \mathbb{U}h\parallel^2=\int_Z |h(T(z))|^2d\bar \mu(z)
\nonumber\\&=&\int_Z | h(z)|^2 d\bar\mu(z)=\parallel h\parallel^2
\end{eqnarray*}
where $\bar \mu$ is an invariant measure. This implies that Koopman operator is unitary.

\item [b).] For any $h\geq 0$, $[\mathbb{U}h](z)\geq 0$ and hence Koopman is a positive operator.

\item [c).]For invertible system $T$, the P-F operator for the inverse system $T^{-1}:Z\to Z$ is given by $\mathbb{P}^*$ and $\mathbb{P}^*\mathbb{P}=\mathbb{P}\mathbb{P}^*=I$. Hence, the P-F operator is unitary.

\item [d).] If the P-F operator is defined to act on the space of densities i.e., $L_1(Z)$ and Koopman operator on space of $L_\infty(Z)$ functions, then it can be shown that the P-F and Koopman operators are dual to each other \footnote{with some abuse of notation we use the same notation for the P-F operator defined on the space of measure and densities.}
\begin{eqnarray*}
&&\left<\mathbb{U} f,g\right>=\int_Z [\mathbb{U} f](z)g(z)dx\nonumber\\&=&\int_Xf(y)g(T^{-1}(y))\left|\frac{dT^{-1}}{dy}\right|dy=\left<f,\mathbb{P} g\right>
\end{eqnarray*}
where $f\in L_{\infty}(Z)$ and $g\in L_1(Z)$ and the P-F operator on the space of densities $L_1(Z)$ is defined as follows
\[[\mathbb{P}g](z)=g(T^{-1}(z))|\frac{dT^{-1}(z)}{dz}|.\]

\item [e).] For $g(z)\geq 0$, $[\mathbb{P}g](z)\geq 0$.

\item [f).] Let $(Z,{\cal B},\mu)$ be the measure space where $\mu$ is a positive but not necessarily the invariant measure of $T:Z\to Z$, then the P-F operator $\mathbb{P}:L_1(Z,{\cal B},\mu)\to L_1(Z,{\cal B},\mu)$  satisfies  following property:

 \[\int_Z [\mathbb{P}g](z)d\mu(z)=\int_Z g(z)d\mu(x).\]\label{Markov_property}
\end{enumerate}
\end{properties}

\section{Finite-dimensional Approximation of the Koopman Operator}\label{section_DMD}

The Koopman operator is an infinite-dimensional operator which governs the evolution of functions on the state-space. Hence, for computation purpose, it is necessary to compute the finite-dimensional approximations of the Koopman operator. To this end, Dynamic Mode Decomposition (DMD) \cite{DMD_schmitt} and Extended Dynamic Mode Decomposition (EDMD) \cite{EDMD_williams} are the most commonly used techniques. In this section, we briefly describe the EDMD algorithm for approximating the Koopman operator. 

Consider 
\begin{eqnarray}
X_p = [x_1,x_2,\ldots,x_M],& X_f = [y_1,y_2,\ldots,y_M] \label{data}
\end{eqnarray}
as snapshots of data set obtained from simulating a discrete time dynamical system $z\mapsto T(z)$ or from an experiment, 
where $x_i\in X$ and $y_i\in X$. We assume $y_i=T(x_i)$. Let $\mathcal{D}=
\{\psi_1,\psi_2,\ldots,\psi_K\}$ be the set of dictionary functions or observables, where $\psi : X \to \mathbb{C}$. Let ${\cal G}_{\cal D}$ denote the span of ${\cal D}$ such that ${\cal G}_{\cal D}\subset {\cal G}$, where ${\cal G} = L_2(X,{\cal B},\mu)$. The choice of dictionary functions are very crucial and it should be rich enough to approximate the leading eigenfunctions of Koopman operator. Define vector valued function $\mathbf{\Psi}:X\to \mathbb{C}^{K}$
\begin{equation}
\mathbf{\Psi}(\boldsymbol{x}):=\begin{bmatrix}\psi_1(x) & \psi_2(x) & \cdots & \psi_K(x)\end{bmatrix}
\end{equation}
In this application, $\mathbf{\Psi}$ is the mapping from physical space to feature space. Any function $\phi,\hat{\phi}\in \mathcal{G}_{\cal D}$ can be written as
\begin{eqnarray}
\phi = \sum_{k=1}^K a_k\psi_k=\boldsymbol{\Psi^T a},\quad \hat{\phi} = \sum_{k=1}^K \hat{a}_k\psi_k=\boldsymbol{\Psi^T \hat{a}}
\end{eqnarray}
for some set of coefficients $\boldsymbol{a},\boldsymbol{\hat{a}}\in \mathbb{C}^K$. Let \[ \hat{\phi}(x)=[\mathbb{U}\phi](x)+r,\]
where $r$ is a residual function that appears because $\mathcal{G}_{\cal D}$ is not necessarily invariant to the action of the Koopman operator. To find the optimal mapping which can minimize this residual, let $\bf K$ be the finite dimensional approximation of the Koopman operator. Then the  matrix $\bf K$ is obtained as a solution of least square problem as follows 
\begin{equation}\label{edmd_op}
\min\limits_{\bf K}\parallel {\bf G}{\bf K}-{\bf A}\parallel_F
\end{equation}
\begin{eqnarray}\label{edmd1}
\begin{aligned}
& {\bf G}=\frac{1}{M}\sum_{m=1}^{M} \boldsymbol{\Psi}({x}_m)^\top \boldsymbol{\Psi}({x}_m)\\
& {\bf A}=\frac{1}{M}\sum_{m=1}^M \boldsymbol{\Psi}({x}_m)^\top \boldsymbol{\Psi}({y}_m),
\end{aligned}
\end{eqnarray}
with ${\bf K},{\bf G},{\bf A}\in\mathbb{C}^{K\times K}$. The optimization problem (\ref{edmd_op}) can be solved explicitly to obtain following solution for the matrix $\bf K$
\begin{eqnarray}
{\bf K}_{EDMD}={\bf G}^\dagger {\bf A}\label{EDMD_formula}
\end{eqnarray}
where ${\bf G}^{\dagger}$ is the psedoinverse of matrix $\bf G$.
DMD is a special case of EDMD algorithm with ${\bf \Psi}(x) = x$.

\section{Koopman Operator Construction for Sparse Data}\label{section_sparse_Koopman}
The finite-dimensional Koopman operator is obtained as a solution to a least-squares problem (\ref{edmd_op}). However, in many experiments, it often is the case that the obtained data-set does not contain enough training examples and thus making the least squares problem ill-posed. In this case, existing algorithms like DMD or EDMD fail to generate acceptable Koopman operators. In fact, in many instances, these algorithms lead to unstable eigenvalues, even though the underlying system is stable \cite{robust_DMD_ACC,robust_DMD_journal}. In this section, we present the main result of this paper, where we propose an algorithm to compute the approximate Koopman operator from a sparse training data-set. In particular, this is achieved in two steps. In the first step we append extra artificial data-points to the training data to make the least-squares problem well-posed and in the second step we account for the \emph{artificiality} of the added data points. 

\subsection{Enrichment of the existing dataset}
The intuition behind addition of extra data-points is the assumption that the underlying dynamical system is at least ${\cal C}^1$ and thus one can argue that \emph{nearby} points are mapped to \emph{nearby} points. 

Let $\bar{X}_p = [{x}_1,{x}_2,\cdots , {x}_M]$ and $\bar{X}_f = [{y}_1,{y}_2,\cdots , {y}_M]$ be the training data from an experiment or a simulation, such that $y_i=T(x_i)$. Corresponding to a training example $x_i$, consider the point ${x}_i+\delta x_i$, where $\parallel \delta x_i\parallel \leq \lambda_X$. Since, $T$ is at least ${\cal C}^1$, 
\begin{eqnarray}
T(x_i+\delta x_i)\approx T(x_i)+\frac{\partial T}{\partial x}\delta x_i = y_i + \delta y_i.
\end{eqnarray}
Since $\parallel \delta x_i\parallel \leq \lambda_X$, $x_i+\delta x_i\in B(x_i,\lambda_X)$, where 
\begin{eqnarray*}
B(x_0,r) = \{x\in\mathbb{R}^N| \parallel x - x_0 \parallel \leq r\}.
\end{eqnarray*}
Again $T$ being $T$ is at least ${\cal C}^1$ implies that $\parallel\frac{\partial T}{\partial x}\parallel \leq \lambda_T$ and hence $T(x_i+\delta x_i)\in B(y_i,\lambda_Y)$, where $\lambda_Y\leq \lambda_X\lambda_T$.

\begin{figure}[htp!]
\centering
\includegraphics[scale=.5]{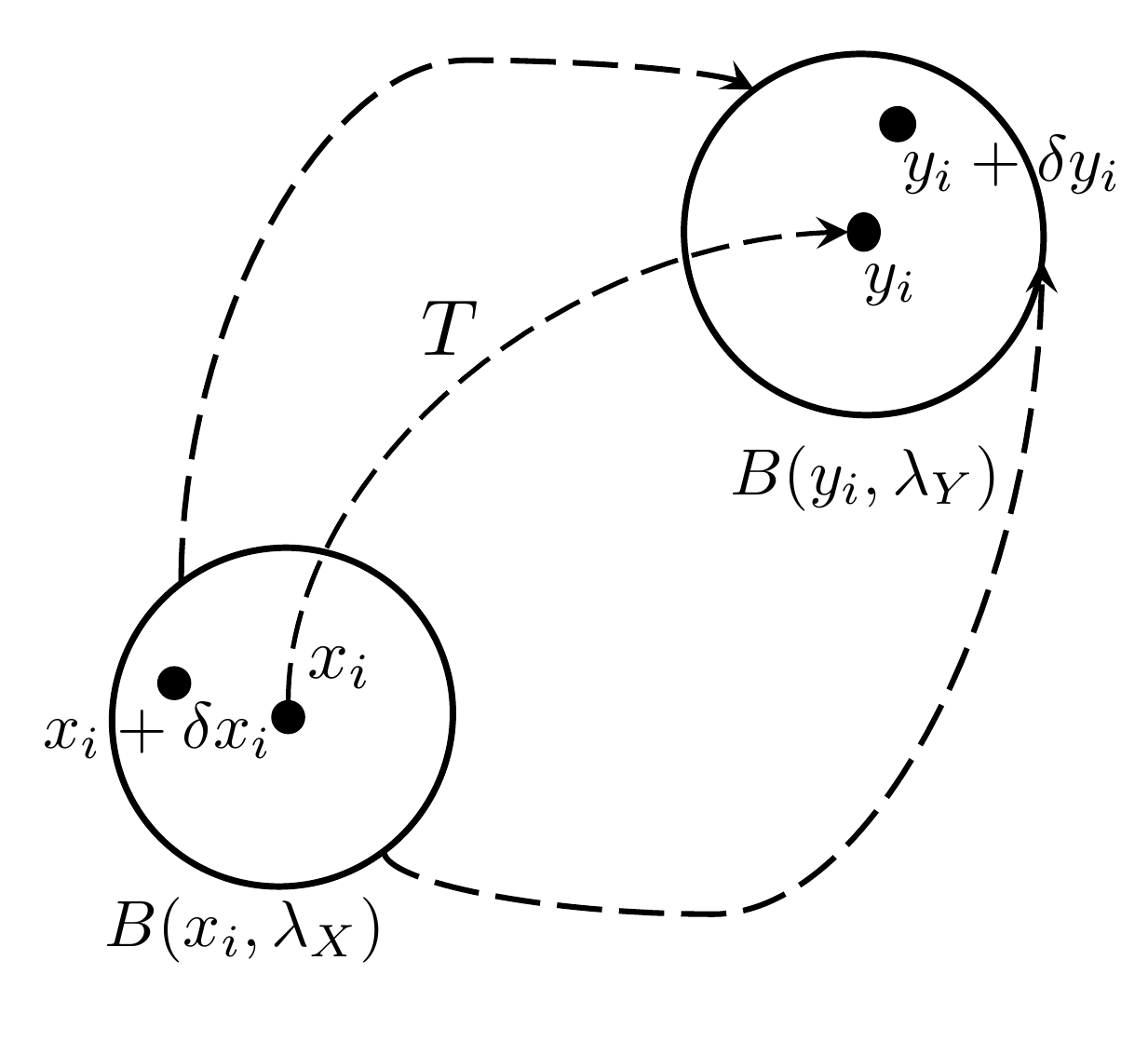}
\caption{Since the map $T$ of the dynamical system is at least ${\cal C}^1$, \emph{nearby} points are mapped to \emph{nearby} points.}\label{continuity}
\end{figure}

Hence, $T$ maps $x\in B(x_i,\lambda_X)$ to $y\in B(y_i,\lambda_Y)$ (Fig. \ref{continuity}) and it is the compactness of the sets $B(x_i,\lambda_X)$ and $B(y_i,\lambda_Y)$ that is used to enrich the existing data set. In particular, to each observed data tuple $(x_i,y_i)$, we augment an extra data point $(x_i+\delta x_i, y_i+\delta y_i)$, where $\delta x_i$ and $\delta y_i$ are random vectors at $x_i$ and $y_i$ respectively, such that $\parallel\delta x_i\parallel\leq\lambda_X$ and $\parallel\delta y_i\parallel\leq \lambda_Y$. Note that, for each data point more than one data point can be augmented, but for clarity, we will discuss the situation where only one extra artificial data point is augmented to each observed data point. Hence, an artificial data set 
\begin{eqnarray}\label{enriched_dataset}
\begin{aligned}
{X}_p &= [{x}_1,\cdots , {x}_M,x_1+\delta x_1, \cdots , x_M + \delta x_M]\\
& = [{x}_1,\cdots ,{x}_{2M}]\\
{X}_f &= [{y}_1,\cdots , {y}_M,y_1+\delta y_1, \cdots , y_M + \delta y_M]\\
& = [{y}_1,\cdots ,{y}_{2M}]
\end{aligned}
\end{eqnarray}
is created with $2(M+1)$ data points. Here $x_{M+i} = x_i + \delta x_i$ and $y_{M+i} = y_i + \delta y_i$.

\subsection{Robust Optimization Formulation}

Of the $2(M+1)$ data points in the enriched data-set, $M+1$ are obtained from an experiment or a simulation, while the other $M+1$ are artificial data points and these points are chosen at random from $B(x_i,\lambda_X)$ and $B(y_i\lambda_Y)$. This implies that, in general, one cannot guarantee that $y=T(x)$, for any particular chosen pair $(x,y)$ where $x\in B(x_i,\lambda_X)$ and $y\in B(y_i,\lambda_Y)$. Hence instead of treating the artificial data-points as exact representation of the underlying system, we view them as noisy observations and the uncertainty acts as an adversary which tries to maximize the residual. Hence, we use robust optimization techniques to compute a Robust Koopman operator from this enriched data-set and the robust optimization problem can be formulated as the following $\min-\max$ optimization problem:

\begin{equation}\label{edmd_robust}
\min\limits_{\bf K}\max_{\delta\in \Delta}\parallel {\bf G}_\delta{\bf K}-{\bf A}_\delta\parallel_F=:\min\limits_{\bf K}\max_{\delta\in \Delta} {\cal J}({\bf K}, {\bf G}_\delta,{\bf A}_\delta)
\end{equation}
where
\begin{eqnarray}
&&{\bf G}_\delta=\frac{1}{2M}\sum_{i=1}^{2M} \boldsymbol{\Psi}({ x}_i)^\top \boldsymbol{\Psi}({x}_i)\nonumber\\
&&{\bf A}_\delta=\frac{1}{2M}\sum_{i=1}^{2M} \boldsymbol{\Psi}({ x}_i)^\top \boldsymbol{\Psi}({ y}_{i}),
\end{eqnarray}
with ${\bf K},{\bf G}_\delta,{\bf A}_\delta\in\mathbb{C}^{K\times K}$.

The robust optimization problem (\ref{edmd_robust}), is in general non-convex because the cost $\cal J$ may not be a convex function of $\delta$. 

\begin{proposition}
The optimization problem (\ref{edmd_robust}) can be approximated as
\begin{equation}\label{edmd_robust_convex}
\min\limits_{\bf K}\max_{\delta{\bf G},\delta{\bf A}\in  {\cal U}}\parallel ({\bf G}+\delta {\bf G}){\bf K}-({\bf A}+\delta {\bf A})\parallel_F
\end{equation}
where $\cal U$ is a compact set in $\mathbb{R}^{K\times K}$.
\end{proposition}
\begin{proof}
From Taylor series expansion we have, ${\bf \Psi}(x_i+\delta x_i) = {\bf \Psi}(x_i) +  {\bf \Psi}'(x_i) \delta x_i+h.o.t.$, where ${\bf \Psi}'(x_i)$ is the first derivative of ${\bf \Psi}(x)$ at $x_i$. Hence,
\begin{eqnarray*}
{\bf G}_{\delta} &\approx& {\bf G} + \frac{1}{2M}\sum_{i=1}^{2M}{\bf \Psi}^\top (x_i)\delta x_i {\bf \Psi}'(x_i)\\
&=& {\bf G} +\delta {\bf G}
\end{eqnarray*}
where $\delta {\bf G} = \frac{1}{2M}\sum_{i=1}^{2M}{\bf \Psi}^\top (x_i)\delta x_i {\bf \Psi}'(x_i)$.

Moreover,
\begin{eqnarray*}
&&\parallel \delta {\bf G} \parallel_F = \parallel \frac{1}{2M}\sum_{i=1}^{2M}{\bf \Psi}^\top (x_i)\delta x_i {\bf \Psi}'(x_i) \parallel_F\\
&\leq& \frac{1}{2M}\sum_{i = 1}^{2M}\parallel {\bf \Psi}^\top (x_i)\delta x_i {\bf \Psi}'(x_i)\parallel_F \\
&\leq&\frac{1}{2M}\sum_{i=1}^{2M}\parallel {\bf \Psi}^\top (x_i)\parallel_F \cdot \parallel\delta x_i \parallel_F \cdot \parallel {\bf \Psi}'(x_i)\parallel_F
\end{eqnarray*}
Hence, $\delta {\bf G}$ belongs to a compact set ${\cal U}_1$. Similarly, one can show ${\bf A}_\delta \approx {\bf A} + \delta {\bf A}$ and $\delta {\bf A}$ belongs to a compact set ${\cal U}_2$. Letting ${\cal U}={\cal U}_1\cup{\cal U}_2$, proves the proposition.
\end{proof}

The above proposition allows us to compute the Koopman operator $\bf K$ as a solution of a robust optimization problem (\ref{edmd_robust_convex}). The optimization problem (\ref{edmd_robust_convex}) has interesting connections with optimization problem involving regularization. In particular, one has the following theorem.

\begin{theorem}
The optimization problem 
\begin{equation}
\min\limits_{\bf K}\max_{\delta{\bf G},\delta{\bf A}\in  {\cal U}}\parallel ({\bf G}+\delta {\bf G}){\bf K}-({\bf A}+\delta {\bf A})\parallel_F
\end{equation}
is equivalent to the following optimization problem
\begin{eqnarray}\label{rob_eqv}
\min\limits_{\bf K}\parallel {\bf G}{\bf K}-{\bf A}\parallel_F+\lambda \parallel {\bf K}\parallel_F\label{regular}
\end{eqnarray}
\end{theorem} 
\begin{proof}
For a $K\times K$ matrix $M=[m_{i,j}]\in\mathbb{R}^{K\times K}$, let ${\cal M}$ denote the vector 
\begin{eqnarray*}
{\cal M}=[m_{1,1},\cdots ,m_{K,1},m_{1,2},\cdots ,m_{K,2},\cdots , m_{K,K}]^\top.
\end{eqnarray*}
This follows from the fact that $\mathbb{R}^{K\times K} \cong \mathbb{R}^{K^2}$. Hence, 
\[\parallel M\parallel_F = \parallel {\cal M}\parallel_2.\]
Again, for two matrices $A$ and $B$, let $A\otimes B$ denote the Kronecker product of $A$ and $B$. Let ${\cal K}$ be the vector form of $\bf K$ and let $\cal A$ and $\delta {\cal A}$ be defined similarly.

Then the min-max optimization problem can be written as
{\small
\begin{eqnarray}\label{opt_2_norm}\nonumber
&&{\cal J} = \min_{\bf K}\max_{\delta{\bf G},\delta{\bf A}\in  \bar \Delta}\parallel ({\bf G}+\delta {\bf G}){\bf K}-({\bf A}+\delta {\bf A})\parallel_F\\ \nonumber
&=& \min_{\cal K}\max_{\delta{\bf G},\delta{\cal A}\in  \bar \Delta}\parallel [({\bf G}+\delta {\bf G})\otimes I_K]{\cal K} - ({\cal A}+\delta {\cal A})\parallel_F\\
&=& \min_{\cal K}\max_{\delta{\bf G},\delta{\cal A}\in  \bar \Delta}\parallel [({\bf G}+\delta {\bf G})\otimes I_K]{\cal K} - ({\cal A}+\delta {\cal A})\parallel_2
\end{eqnarray}
}
where $I_K$ is the $K\times K$ identity matrix.
Writing ${\bf G}\otimes I_K$ as $\hat{G}$ and $\delta {\bf G}\otimes I_K$ as $\delta\hat{G}$, the optimization problem (\ref{opt_2_norm}) can be written as
\begin{eqnarray}\label{opt_2_norm_modified}
{\cal J} = \min_{\cal K}\max_{\substack{\delta{\hat{G}}\in \Pi_K\bar{\Delta}\\ \delta{\cal A}\in  \bar \Delta}}\parallel ({\hat{G}}+\delta {\hat{G}}){\cal K} - ({\cal A}+\delta {\cal A})\parallel_2
\end{eqnarray}

Fix ${\bf K}\in\mathbb{R}^{K\times K}$ and let
\begin{eqnarray}\label{worst_residual_geq}
r = \max_{\substack{\delta{\hat{G}}\in \Pi_K\bar{\Delta}\\ \delta{\cal A}\in  \bar \Delta}}\parallel ({\hat{G}}+\delta {\hat{G}}){\cal K} - ({\cal A}+\delta {\cal A})\parallel_2
\end{eqnarray}
be the worst-case residual. Then,
\begin{eqnarray}\nonumber\label{worst_residual_leq}
r &&\leq \max_{\substack{\delta{\hat{G}}\in \Pi_K\bar{\Delta}\\ \delta{\cal A}\in  \bar \Delta}}\parallel{\hat G}{\cal K} - {\cal A} \parallel_2 + \parallel\delta {\hat G}{\cal K} - \delta {\cal A} \parallel_2 \leq \parallel{\hat G}{\cal K} - {\cal A} \parallel_2 + \lambda \parallel{\cal K} - \mathds{1} \parallel_2\\
&& \leq \parallel{\hat G}{\cal K} - {\cal A} \parallel_2 + \lambda \sqrt{\parallel {\cal K}\parallel_2^2 + K} = \parallel {\bf G} {\bf K} - {\bf A} \parallel_F + \lambda \sqrt{\parallel {\bf K}\parallel_F^2 + K}
\end{eqnarray}

Again, choose $[\delta {\cal G} \quad \delta {\cal A}]$ as 
\[[\delta {\cal G} \quad \delta {\cal A}]= \frac{\lambda u}{\sqrt{\parallel {\cal K}\parallel_2^2 + K}}[{\cal K}^\top\quad K],\]
where 
\begin{equation}
  u =
    \begin{cases}
      \frac{{\cal G}{\cal K}-{\cal A}}{\parallel {\cal G}{\cal K}-{\cal A} \parallel}, \textnormal{ if } {\cal G}{\cal K}\neq{\cal A}\\
      \textnormal{any unit norm vector otherwise.}
    \end{cases}       
\end{equation}
Then,
\begin{eqnarray}\nonumber\label{worst_residual_geq1}
r &=& \max_{\substack{\delta{\hat{G}}\in \Pi_K\bar{\Delta}\\ \delta{\cal A}\in  \bar \Delta}}\parallel ({\hat{G}}{\cal K}-{\cal A}) + (\delta {\hat{G}}{\cal K}-\delta {\cal A})\parallel_2\\ \nonumber
&=& \max_{\substack{\delta{\hat{G}}\in \Pi_K\bar{\Delta}\\ \delta{\cal A}\in  \bar \Delta}}\parallel ({\hat{G}}{\cal K}-{\cal A}) + \lambda (\frac{{\cal G}{\cal K}-{\cal A}}{\parallel {\cal G}{\cal K}-{\cal A} \parallel}{\cal K}^\top {\cal K} +K\frac{{\cal G}{\cal K}-{\cal A}}{\parallel {\cal G}{\cal K}-{\cal A} \parallel})\parallel_2\\ \nonumber 
&\geq& \parallel ({\hat{G}}{\cal K}-{\cal A})\parallel_2 +\lambda \parallel (\frac{{\cal G}{\cal K}-{\cal A}}{\parallel {\cal G}{\cal K}-{\cal A} \parallel}{\cal K}^\top {\cal K} +K\frac{{\cal G}{\cal K}-{\cal A}}{\parallel {\cal G}{\cal K}-{\cal A} \parallel})\parallel_2\\ 
&\geq& \parallel ({\hat{G}}{\cal K}-{\cal A})\parallel_2 + \lambda \sqrt{{\cal K}^\top {\cal K}+K} = \parallel {\bf G} {\bf K} - {\bf A} \parallel_F + \lambda \sqrt{\parallel {\bf K}\parallel_F^2 + K}
\end{eqnarray}

Hence, from (\ref{worst_residual_leq}) and (\ref{worst_residual_geq1}), the worst case residual is 
\begin{eqnarray}\label{worst_res_eqv}
r = \min_{{\bf K}}\parallel {\bf G}{\bf K} - {\bf A} \parallel_F + \lambda \sqrt{\parallel {\bf K}\parallel_F^2 +K}.
\end{eqnarray}
Since, K is a constant, the $\bf K$ that minimizes $r$ in (\ref{worst_res_eqv}) is the same $\bf K$ that minimizes
\[\parallel {\bf G}{\bf K} - {\bf A} \parallel_F + \lambda \parallel {\bf K}\parallel_F .\]
\end{proof}

The above theorem allows computation of the approximate Koopman operator as a solution of an optimization problem with a regularization term. In particular, the approximate Koopman operator can be obtained as a solution of the following optimization problem
\begin{eqnarray}\label{robust_opt}
\parallel {\bf G} {\bf K} - {\bf A} \parallel_F + \lambda \parallel {\bf K}\parallel_F.
\end{eqnarray}
where
\begin{eqnarray}
\begin{aligned}
&{\bf G}=\frac{1}{2M}\sum_{i=1}^{2M} \boldsymbol{\Psi}({ x}_i)^\top \boldsymbol{\Psi}({x}_i)\\
&{\bf A}=\frac{1}{2M}\sum_{i=1}^{2M} \boldsymbol{\Psi}({ x}_i)^\top \boldsymbol{\Psi}({ y}_{i}).
\end{aligned}
\end{eqnarray}
\begin{algorithm}[htp!]
\caption{Sparse Koopman Algorithm}
\begin{enumerate}
\item{To the existing data set $D = [x_1,\cdots, x_M]$ add new data points $\tilde{x}_i = x_i+\delta x_i$, where $\parallel \delta x_i\parallel \leq c$.}
\item{Form the enriched data set $\bar{D} = [x_1,\cdots, x_M, \tilde{x}_1,\cdots, \tilde{x}_M]$.}
\item{Form the sets $X_p = [x_1,\cdots, x_{M-1}, \tilde{x}_1,\cdots, \tilde{x}_{M-1}]$ and $X_f = [x_2,\cdots, x_M, \tilde{x}_2,\cdots, \tilde{x}_M]$.}
\item{Fix the dictionary functions ${\bf \Psi} = [\psi_1,\cdots , \psi_K]$.}
\item{Solve the optimization problem to obtain the approximate Koopman operator $\bf K$
\begin{eqnarray*}
\parallel {\bf G} {\bf K} - {\bf A} \parallel_F + \lambda \parallel {\bf K}\parallel_F.
\end{eqnarray*}
where
\begin{eqnarray*}
&&{\bf G}=\frac{1}{2M}\sum_{i=1}^{2M} \boldsymbol{\Psi}({ x}_i)^\top \boldsymbol{\Psi}({x}_i)\nonumber\\
&&{\bf A}=\frac{1}{2M}\sum_{i=1}^{2M} \boldsymbol{\Psi}({ x}_i)^\top \boldsymbol{\Psi}({ y}_{i}).
\end{eqnarray*}
}
\end{enumerate}
\label{algo}
\end{algorithm}

\begin{remark}
The optimization problem (\ref{worst_res_eqv}) can also be formulated as a Second Order Cone Problem (SOCP) as follows
\begin{eqnarray}
\begin{aligned}
&\min \qquad \qquad \quad \theta\\
&\textnormal{subject to } \parallel {\hat{G}} {\cal K} - {\hat{A}} \parallel \leq \theta - \tau,\\
&\qquad \qquad \quad \Bigg\Vert\begin{pmatrix}
{\cal K}\\
K
\end{pmatrix}\Bigg\Vert \leq \tau.
\end{aligned}
\end{eqnarray}
\end{remark}

The Sparse Koopman Learning Algorithm appends artificial data points to the training data sets and in the above discussion to each obtained data point we appended one artificial data point. However, to each obtained data point, one can append more than one artificial data point. In particular we have the following theorem.

\begin{theorem}
Let $x_{t+1}=T(x_t)$ be a dynamical system with $x_t\in\mathbb{R}^N$ and let \[\mathbf{\Psi}(\boldsymbol{x})=\begin{bmatrix}\psi_1(x) & \psi_2(x) & \cdots & \psi_K(x)\end{bmatrix}\] be a set of dictionary functions. Let $\{x_i\}$, $i=1,\cdots , {M+1}$ be the obtained data set. Then at each $x_i$ we can append $l$ artificial data points such that $l = \min \{N, \textnormal{rank}(J_i)\}$, where $J_i=\frac{\partial {\mathbf{\Psi}}}{\partial x}\rvert_{x_i}$ is the Jacobian of $\bf \Psi$ evaluated at $x_i$.
\end{theorem}
\begin{proof}
Let $X=\{x_1,\cdots , x_{M+1}\}$ be the set of obtained data-points and consider a single data point $x_p\in X$. We assume that the dynamical system $x_{t+1}=T(x_t)$ evolves on the state space $M$, where by a slight abuse of notation we identify a point on the manifold $M$ by its vector representation $x_p$ in $\mathbb{R}^N$. Now, at each point $x_p$, we consider the tangent space $T_{x_p}\mathbb{R}^N$. Since the dynamical system evolves on $\mathbb{R}^N$, the tangent space $T_{x_p}\mathbb{R}^N\cong \mathbb{R}^N$. Hence at each point $x_p$ one can construct $N$ independent vectors $\delta x_i$, $i=1, \cdots , N$, such that $\parallel \delta x_i\parallel < \lambda$. 

Now, using taylor series expansion for $\bf \Psi$ for points around $x_p$, we have 
\[{\bf \Psi}(x_p+\delta x_i) = {\bf \Psi}(x_p) + \frac{\partial {\bf \Psi}}{\partial x}\rvert_{x_p}\delta x_i + h.o.t.\]
Hence the vectors $x_p+\delta x_i$ are lifted to the tangent space $T_{{\bf \Psi}(x_p)}\mathbb{R}^K$ by the dictionary function $\bf \Psi$. Note that this construction is exactly same as push-forward of tangent vectors from $T_x{\cal M}$ to $T_y{\cal N}$ by a differentiable map $f:{\cal M}\to {\cal N}$, such that $y=f(x)$. 

Let $J_p=\frac{\partial {\bf \Psi}}{\partial x}\rvert_{x_p}$, such that rank$(J_p)=r$. If $r=N$, then $J_p$ is injective and hence the independent vectors $x_p+\delta x_i$ are mapped to independent vectors. In this case, we can add $N$ artificial data-points around $x_p$ to the original data-set and these artificial data-points are exactly $x_p+\delta x_i$. This is because if we choose $N+1$ artificial data-points, then the data-points will be linearly dependent and hence adding them to the data-set will not change the rank of $\bf G$ or $\bf A$ matrices in the optimization problem (\ref{robust_opt}) and thus will not help is computation of the Koopman operator. Again, when $r<N$, only $r$ independent vectors from $T_{x_p}\mathbb{R^N}$ will be mapped to $r$ independent vectors in $T_{{\bf \Psi}(x_P)\mathbb{R}^K}$ and in this case one can add only $r$ artificial data-points to the obtained data-set around $x_p$.

\end{proof}

\section{Sparse Koopman Learning Algorithm and Learning Performance}\label{section_VC_analysis}

Computation of Koopman operator (EDMD algorithm) amounts to solving a least squares problem and if the data-set is small, the least squares solution is often ill-posed \cite{tikhonov2013numerical}. Hence, increasing the number of data points for training the algorithm is always favourable for efficient learning. 
The proposed Sparse Koopman Learning algorithm artificially increases the number of data points used to train the Koopman operator and this is possible because the dynamical systems from which the data is obtained is assumed to be at least $\mathcal{C}^1$. Hence, one can use the continuity argument that nearby points are mapped to nearby points, thus enriching the limited data-set artificially and the Koopman operator is obtained as a solution to a regularized least squares problem. In this section, we establish why the Sparse Koopman Algorithm works by using the concepts of Vapnik-Chervonenkis (VC) dimension and bias-variance trade-off \cite{vapnik_book,mostafa_learning_book}. 

\subsection{Artificial Data Points and Vapnik-Chervonenkis Dimension}
The Sparse Koopman Algorithm consists of two parts:
\begin{itemize}
\item{addition of artificial data points to the existing data-set, and}
\item{solving a regularized least squares problem.}
\end{itemize}
Both of these play important roles in obtaining an efficient Koopman operator.We analyze the role of addition of artificial data points by the Vapnik-Chervonenkis (VC) dimension. For simplicity, we define VC dimension for classification problem and later discuss how the concept can be generalized to linear regression problems.

\begin{definition}[Shattering and VC Dimension \cite{vapnik_book,mostafa_learning_book}]
Let ${\cal H}$ be a class of $\{\pm 1\}$-valued functions on space $\cal X$ . We say a set of $m$ points $\{x_1, . . . , x_m\} \subset \mathcal{X}$ is shattered by $\cal H$ if all possible $2^m$ binary labellings of the points can be realized by functions in $\cal H$. The VC dimension of $\cal H$, denoted by $d_{VC}(\mathcal{H})$, is the cardinality of the largest set of points in $\cal X$ that can be shattered by $\cal H$. If $\cal H$ shatters arbitrarily large sets of points in $\cal X$, then $d_{VC}(\mathcal{H})=\infty$ 
\end{definition}

\begin{figure}[htp!]
\centering
\includegraphics[scale=.5]{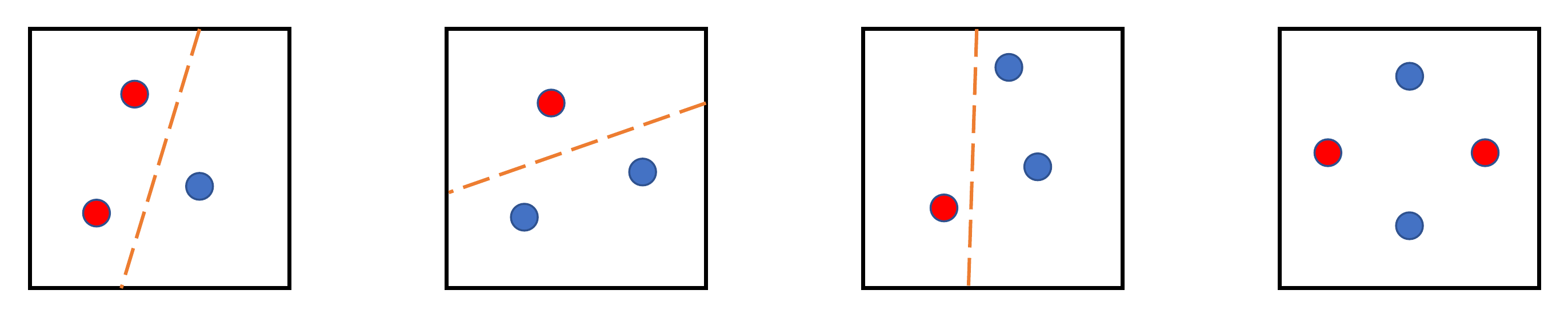}
\caption{All the three points in the left three figures can be shattered by a line in $\mathbb{R}^2$. However, in $\mathbb{R}^2$, 4 points, in general can not be shattered by a line. Hence the VC dimension is 3.}\label{vc_shattering}
\end{figure}

For example, consider points in 2-dimensional space, that is, $\mathbb{R}^2$ and consider the hypothesis set consisting of linear classifiers (Fig. \ref{vc_shattering}). In this case, any 3 points, which are not colinear, can be classified in red and blue categories. However, if there are 4 points, then the configuration shown in rightmost box of Fig. \ref{vc_shattering} can not be classified by a linear classifier. Hence, the VC dimension of linear classifiers on $\mathbb{R}^2$ is 3. In general, the VC dimension of linear classifiers on $\mathbb{R}^d$ is $(d+1)$.

The above definition of VC dimension for dyadic functions can be extended to regression problems (or any real valued function) as follows. Consider real-valued functions $\{f(x,\theta)\}$ (here $\theta$ is the space of parameters) taking values $y$. With this one can construct dyadic functions

\begin{equation}\label{dyadic_construction}
  f_y(x,\theta) =
    \begin{cases}
      1 & \text{if $f(x,\theta)-y>0$}\\
      -1 & \text{if $f(x,\theta)-y\leq 0$}.
    \end{cases}       
\end{equation}
With this, the VC dimension of $f_y(x,\theta)$ is defined as before. 

\begin{figure}[htp!]
\centering
\includegraphics[scale=.6]{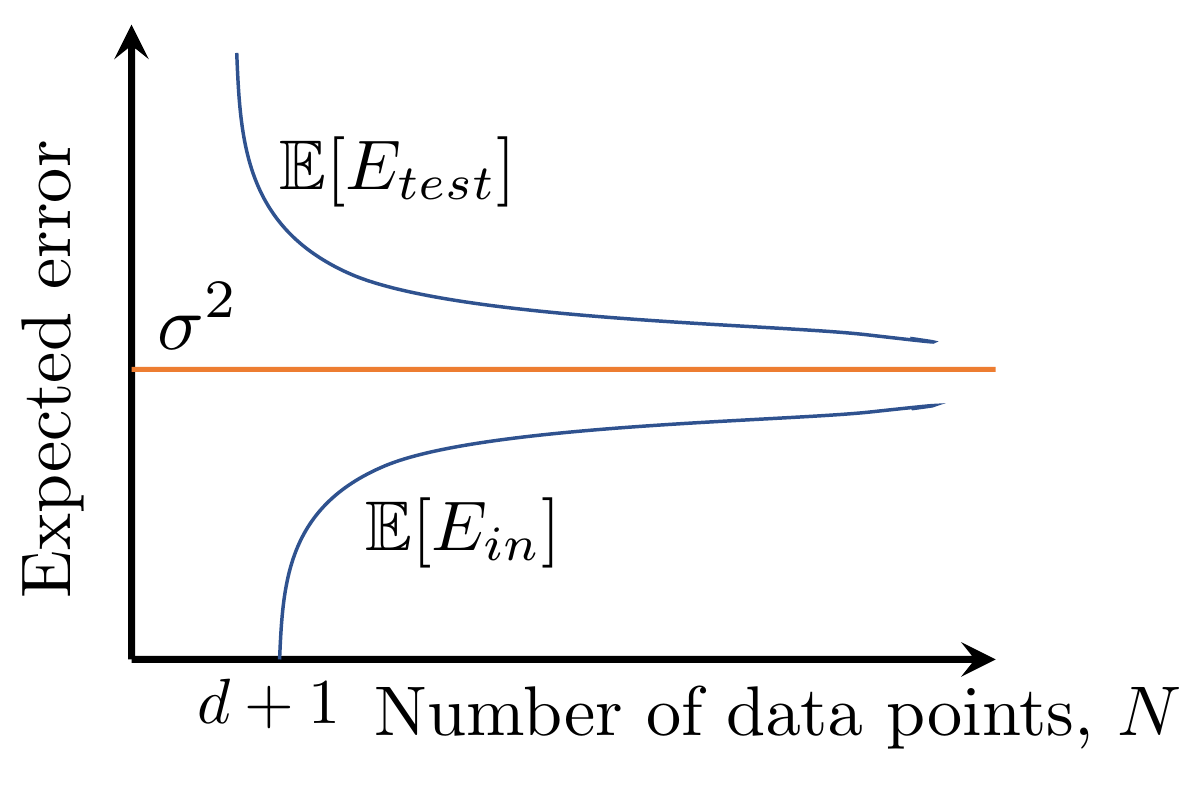}
\caption{Learning curve for linear regression}\label{learning_curve}
\end{figure}

Now, let $\{(x_1,y_1), (x_2,y_2),\cdots , (x_N,y_N)\}$ be a data set with $x_i\in\mathbb{R}^d$, such that $y=w^{*\top}+\epsilon$, where $\epsilon$ is the noise term with zero mean and variance $\sigma^2$. The optimal weight vector $w^{*\top}$ is given by $w^{*\top} = X^\dagger y$ where $X^\dagger$ is the pseudo-inverse of $X = [x_1,x_2,\cdots ,x_N]^\top$ and $y = [y_1,y_2,\cdots ,y_N]^\top$. With this the expected value of the in-sample error $(E_{in})$ is given by \cite{mostafa_learning_book}
\begin{eqnarray}\label{train_error}
\mathbb{E}[E_{in}] = \sigma^2(1-\frac{d}{N})
\end{eqnarray}
and the expected value of the test error is 
\begin{eqnarray}\label{test)_error}
\mathbb{E}[E_{test}] = \mathbb{E}[E_{in}] + O\left(\frac{d}{N}\right),
\end{eqnarray}
where $\mathbb{E}[\cdot]$ denotes the expectation operator. With this, the best linear fit has expected error $\sigma^2$ and this fit is attained as the number of data-points $N$ becomes large. This is shown in Fig. \ref{learning_curve}.

Hence it is always desired to have a large number of training sample to obtain good linear fit, thus resulting in a model which generalizes well to test data. However, in many real life situations, it may not be possible to obtain large training data-sets and the goal of this paper is to propose an algorithm that can tackle such situations. To this end, under the assumption that the underlying dynamical system map is at least $\mathcal{C}^1$, we proposed a prescription which enlarges the existing training data-set by appending artificial data points. Hence, by (\ref{train_error}), our proposed algorithm reduces the training error and in the process a more accurate Koopman operator is obtained. However, our algorithm uses ideas from robust optimization and it results in a regularized linear regression problem (\ref{robust_opt}) and this poses a problem for VC analysis. In particular, with change in the regularization parameter $\lambda$, the learning algorithm changes, but the hypothesis set remains the same and hence the VC dimension remains the same. However, as the regularization parameter $\lambda$ is increased, it makes the weights of the linear regression model more constrained. In particular, the unconstrained optimization problem (\ref{robust_opt}) can be recast as a constrained optimization problem as
\begin{eqnarray}
\begin{aligned}
& \min \qquad \parallel {\bf G} {\bf K} - {\bf A}\parallel_F\\
& \textnormal{subject to } \parallel {\bf K}\parallel_F \leq C,
\end{aligned}
\end{eqnarray}
where $C$ is related to the regularization parameter $\lambda$, such that when $\lambda$ increases, $C$ decreases and vice versa. Hence, when $\lambda$ is increased, the optimization variable ${\bf K}$ is being constrained more and more and hence correspond to a smaller model (the set of allowable weights decreases in size) and thus we expect better generalization for a small increase in $\mathbb{E}[E_{in}]$, even though the VC dimension remains same. In such a situation, for regularized linear regression, a heuristic concept of ``effective VC dimension" is used instead of normal VC dimension \cite{mostafa_learning_book}. However, there are multiple definitions of ``effective VC dimension" in literature \cite{mostafa_learning_book}, but they all establish the fact that if the number of training data-points are increased the performance of the learning algorithm improves \cite{mostafa_learning_book}. Thus addition of extra data points to the original training data do improve the efficiency of the Sparse Koopman Algorithm. However, the artificial data points are considered as noisy observations and it degrades the performance of the algorithm and to take care of the noisy observations, we use robust optimization techniques. In particular, we use regularized least squares to account for the noisy observations and the role of regularization in the Sparse Koopman Algorithm can be studied via the Bias-Variance Trade-off.

\subsection{Regularization and Bias-Variance Trade-off}

The VC dimension depends on the hypothesis set $\cal H$ and it shows that the choice of $\cal H$ leads to a trade-off between the approximation of the target function on the training set and the performance of the obtained function on the test data. In particular, if $\cal H$ is too simple, it may lead to a large training error and if the hypothesis set is too complex, it may to lead to overfitting and thus lead to large test error. This is known as Bias-Variance trade-off \cite{mostafa_learning_book}. The intuition of bias-variance trade-off is explained in Figure \ref{bias_variance}.  Usually, with a highly complex model, it is possible to fit the training data as closely as possible.  In this case, the training error is extremely small and the model is said to have a low bias.  However, the highly complex model may not generalize well to the test data, thus making the test error large.  This is due to the overfitting of the training data.  The complex model, which overfits the training data and produces high test error, is said to have high variance.  This situation is often reversed if the model considered is fairly simple.

\begin{figure}[htp!]
\centering
\includegraphics[scale=.3]{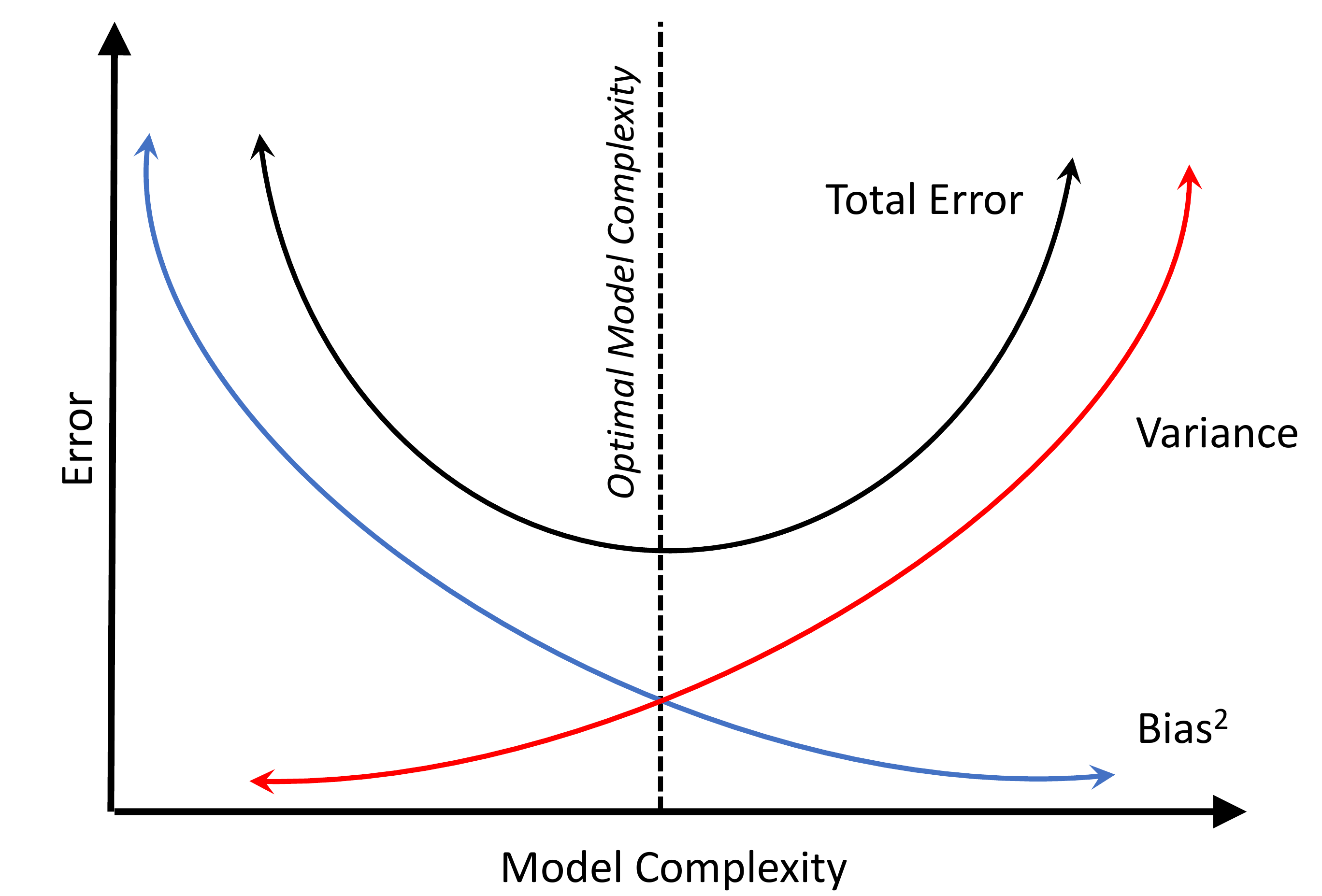}
\caption{Bias-Variance Trade-off and model complexity.}\label{bias_variance}
\end{figure}

For a linear regression problem, it can be shown that the error on a test data-point can be decomposed into a sum of bias squared and variance \cite{mostafa_learning_book}. Balancing this bias-variance trade-off is at the heart of developing a successful model and the regularization parameter $\lambda$ acts as the parameter which achieves this balance. In particular, increase in the regularization parameter $\lambda$ puts more emphasis on the norm of the parameters in the regularized least square optimization problem (\ref{robust_opt}) and thus shrinks the parameters towards zero. This leads to a higher bias, which is undesirable, but on the other hand, it reduces the variance and thus there exists a $\lambda$ which achieves the perfect balance between the bias and variance, leading to the best model.

In the Sparse Koopman Algorithm, the regularization of the standard EDMD algorithm achieves this bias-variance trade-off by acting against overfitting the Koopman operator to the noisy artificial data-points and thus resulting in an efficient Koopman operator computation. Hence, to summarize, in the Sparse Koopman Algorithm, addition of extra data-points help to make the least square optimization problem well-posed and the regularization help against overfitting the noisy data and yield a more accurate Koopman operator.

\section{Design of Robust Predictor}\label{section_predictor}
The Koopman operator generates a linear system in a higher dimensional space, even if the underlying system is linear. The linearity of the operator enables the design of linear predictors for nonlinear systems. The following is presented briefly for the self-containment of the paper and for details the readers are referred to \cite{korda_mezic_predictor}. Let $\{x_0,\ldots,x_M\}$ be the training data-set and $\bf K$  be the finite-dimensional approximation of the transfer Koopman operator obtained using algorithm ${\bf 1}$. Let $\bar x_0$ be the initial condition from which the future is to be predicted. The initial condition from state space is mapped to the feature space using the same choice of basis function used in the robust approximation of Koopman operator i.e., \[\bar x_0\implies {\bf \Psi}(\bar x_0)^\top=: {\bf z}\in \mathbb{R}^K.\] This initial condition is propagated using Koopman operator as \[{\bf z}_n={\bf K}^n{\bf z}.\]
The predicted trajectory in the state space is then obtained as 
\[\bar x_n=C {\bf z}_n\]
where matrix $C$ is obtained as the solution of the following least squares problem
\begin{eqnarray}\label{C_pred}
\min_C\sum_{i = 1}^M \parallel x_i - C \boldsymbol \Psi (x_i)\parallel_2^2
\end{eqnarray}

\section{Simulations}\label{section_simulation}

In this section, we demonstrate the efficiency of the proposed algorithm on three different dynamical systems. In particular, we construct the Koopman operator for a linear system, a non-linear system and a system governed by a Partial Differential Equation (PDE). 
\subsection{Network of Coupled Oscillators}
Consider a network of coupled linear oscillators given by
\begin{eqnarray}\label{coup_osc}
\ddot{\theta}_k &=& -\mathcal{L}_k\theta - d\dot{\theta}_k, \quad k = 1,\cdots , N
\end{eqnarray}
where $\theta_k$ is the angular position of the $k^{th}$ oscillator, $N$ is the number of oscillators, $\mathcal{L}_k$ is the $k^{th}$ row of the Laplacian $\mathcal{L}$ and $d$ is the damping coefficient. The Laplacian $\cal L$ is chosen such that the network is a ring network with 20 oscillators (Fig. \ref{fig_oscillator}).

\begin{figure}[htp!]
\centering
\includegraphics[scale=.6]{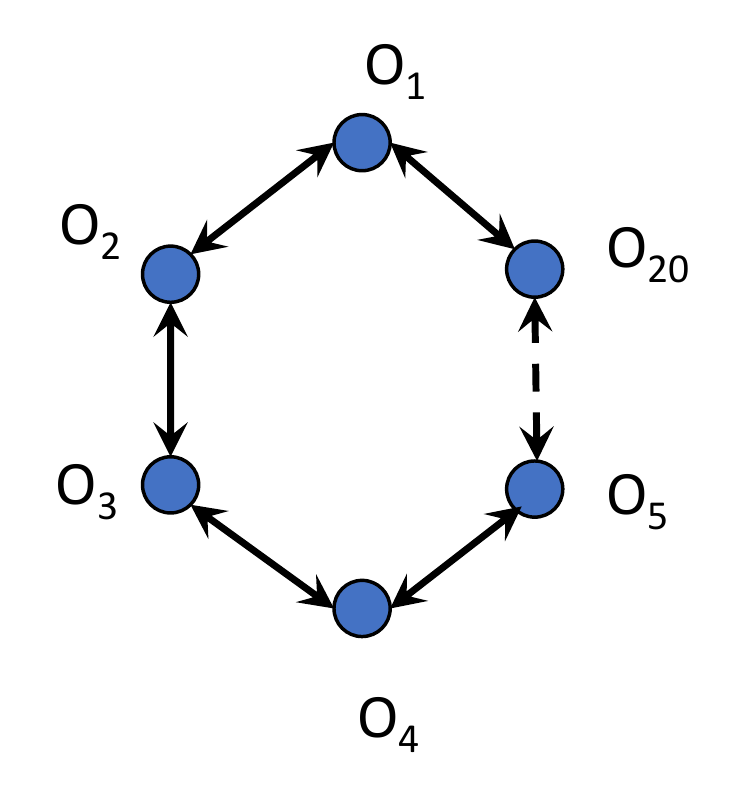}
\caption{Ring network of 20 linear oscillators.}\label{fig_oscillator}
\end{figure}

In these sets of simulations, the damping coefficient $d$ has been assumed the same for all the oscillators and is set equal to $0.4$. Data for all the states were collected for 100-time steps, with sampling time $\delta t = 0.01$ seconds and since the system is linear, linear basis functions were used for computation of the Koopman operator. 
\begin{figure}[htp!]
\centering
\subfigure[]{\includegraphics[scale=.31]{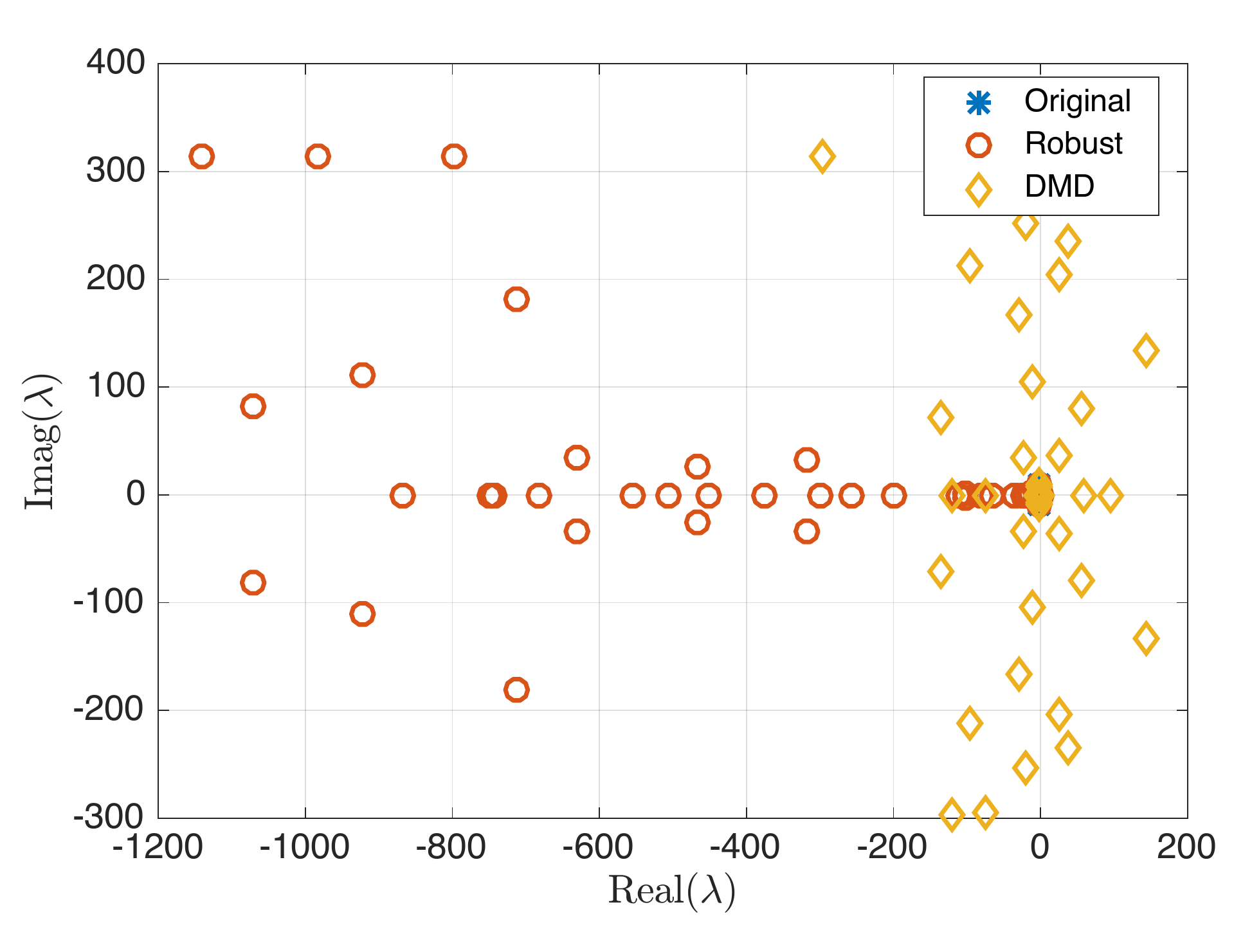}}
\subfigure[]{\includegraphics[scale=.31]{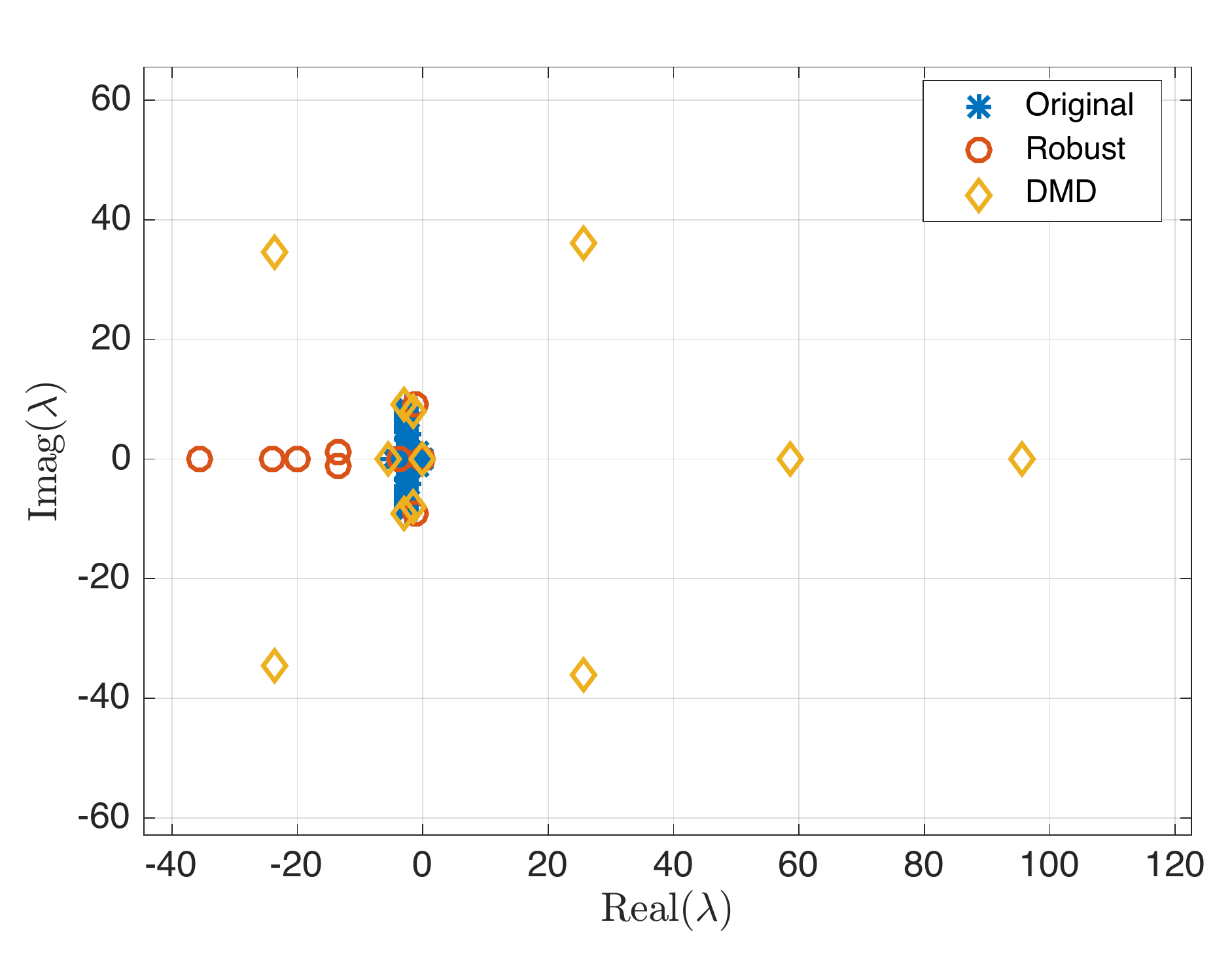}}
\caption{(a) Eigenvalue obtained using normal DMD on original training data and Robust DMD on enriched data set. (b) Dominant eigenvalues.}\label{eig_osc}
\end{figure}
The first 15-time steps data was used for training the Koopman operators. Normal DMD on the 15 data points yields positive eigenvalues with a significant real part, as shown in Fig. \ref{eig_osc}. For the Robust identification of Koopman operator, the original data set was enriched by adding 30 artificial data points and Robust DMD formulation (algorithm \ref{algo}) yields a much better approximation of the eigenvalues for the original system. The eigenvalues obtained using normal DMD and Robust DMD are shown in Fig. \ref{eig_osc}, wherein Fig. \ref{eig_osc}a the complete spectrum is plotted and in Fig. \ref{eig_osc}b the dominant eigenvalues are shown. 

\begin{figure}[htp!]
\centering
\subfigure[]{\includegraphics[scale=.315]{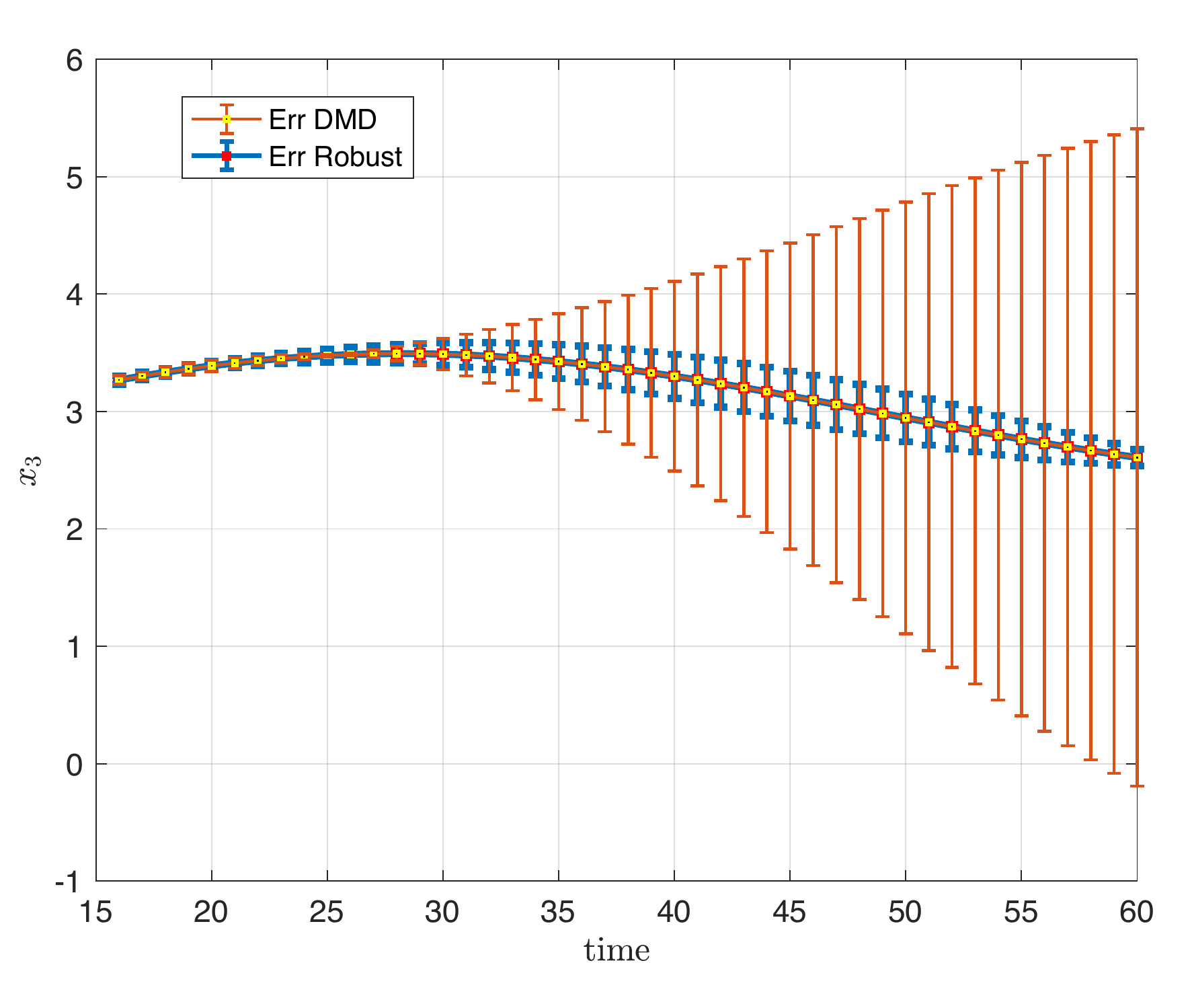}}
\subfigure[]{\includegraphics[scale=.315]{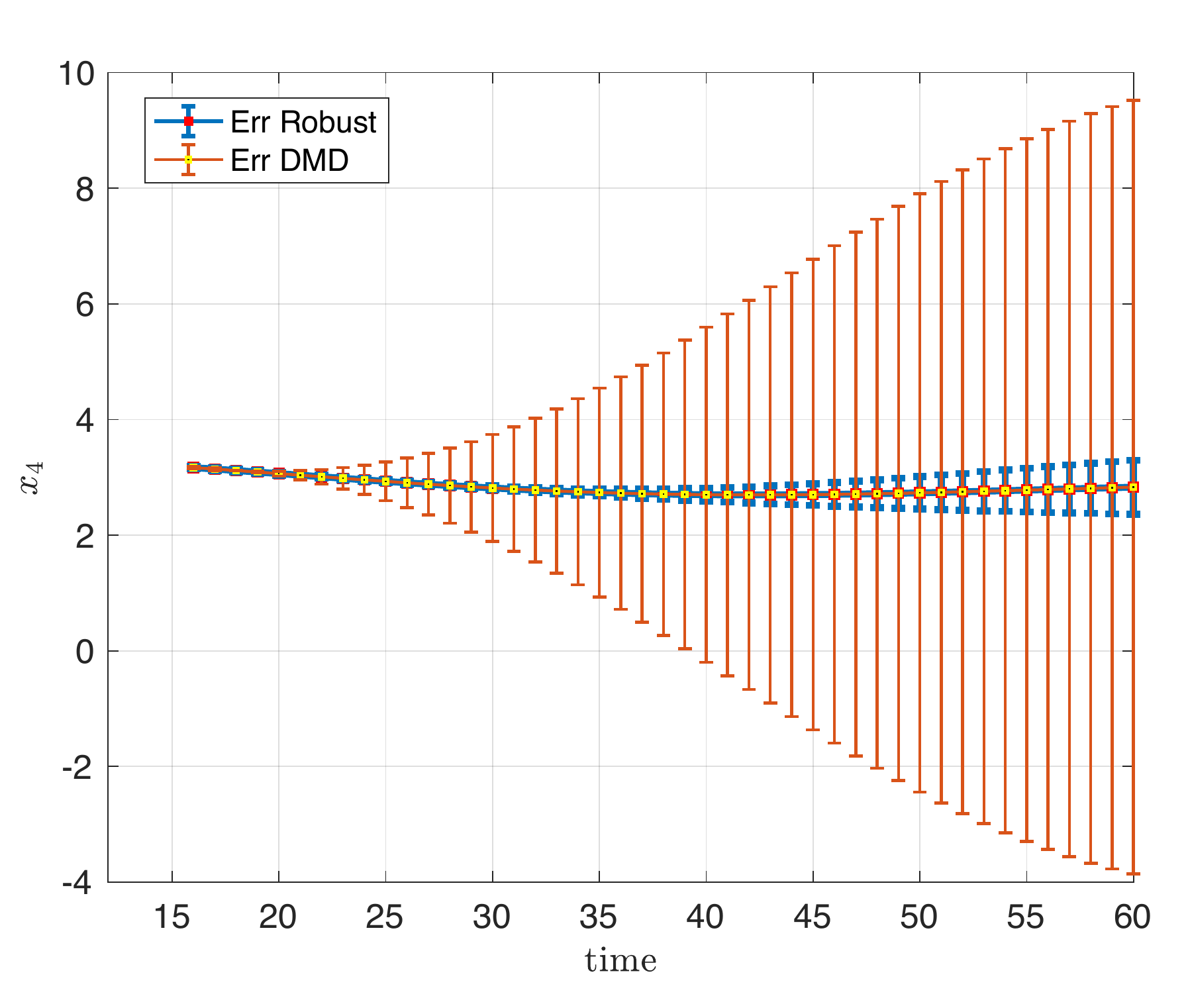}}
\caption{(a) Errors in the prediction of position of oscillator 3. (b) Errors in the prediction of the position of oscillator 4.}\label{err_osc}
\end{figure}
As mentioned earlier, data were obtained for 100 times steps and the first 15 time steps were used for training the Koopman operator. Koopman operators thus obtained was used to predict the next 45 time steps and was used to compare the error. The errors in the prediction of the positions of oscillators 3 and 4, using both normal DMD and Robust DMD, are shown in Fig. \ref{err_osc}a and Fig. \ref{err_osc}b respectively. It can be observed that Robust DMD formulation generates much smaller error compared to normal DMD. In fact, this was expected, since Robust DMD with enriched data-set approximates the eigenspectrum much better compared to normal DMD.

\subsection{Stuart-Landau Equation}
The nonlinear Stuart-Landau equation on a complex function $z(t) = r(t)\exp (i\theta(t))$ is given by 
\begin{eqnarray}\label{stu_lan}
\dot{z} = (\mu + i\gamma)z - (1+ i\beta)|z|^2z ,
\end{eqnarray}
where $i$ is the imaginary unit. The solution of (\ref{stu_lan}) evolves on the limit cycle $|z| = \sqrt{\mu}$. Hence, the continuous time eigenvalues lie on the imaginary axis. The discretized version of (\ref{stu_lan}) is

{\small
\begin{eqnarray}\label{stu_lan_dis}
\begin{pmatrix}
r_{t+1}\\
\theta_{t+1}
\end{pmatrix} = \begin{pmatrix}
r_t + (\mu r_t -r_t^3)\delta t\\
\theta_t + (\gamma - \beta r_t^2)\delta t
\end{pmatrix}
\end{eqnarray}
}

The set of dictionary functions were chosen as
\begin{eqnarray}\label{stuart_dic}
{\bf \Psi}(\theta_t) = \begin{pmatrix}
e^{-10i\theta_t} & e^{-9i\theta_t} & \cdots & e^{9i\theta_t} & e^{10i\theta_t}
\end{pmatrix}
\end{eqnarray}
 and data was collected for 150 time steps, with $\delta t = 0.01$ and initial condition $(1,\pi)$. 

\begin{figure}[htp!]
\centering
\subfigure[]{\includegraphics[scale=.35]{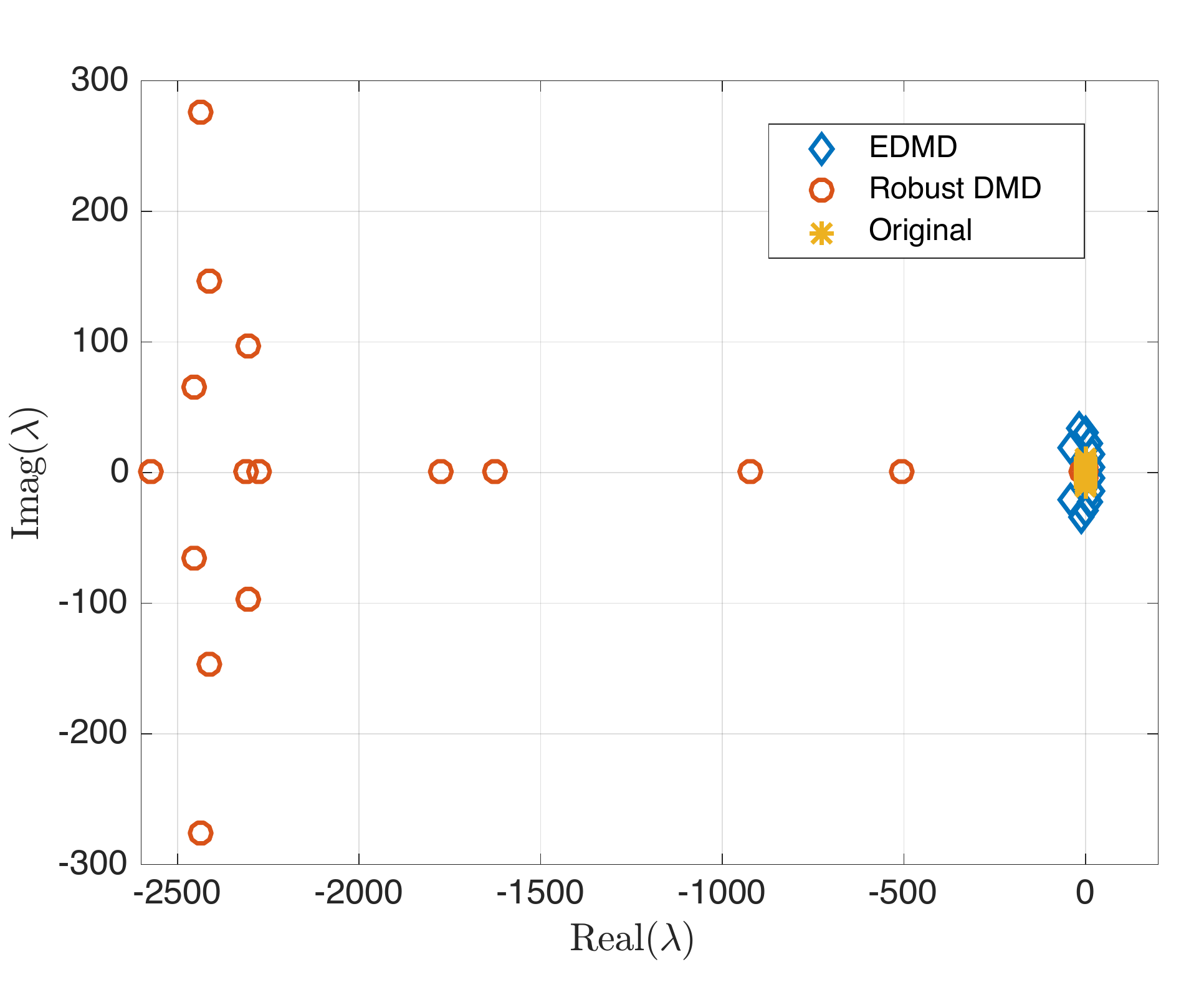}}
\subfigure[]{\includegraphics[scale=.35]{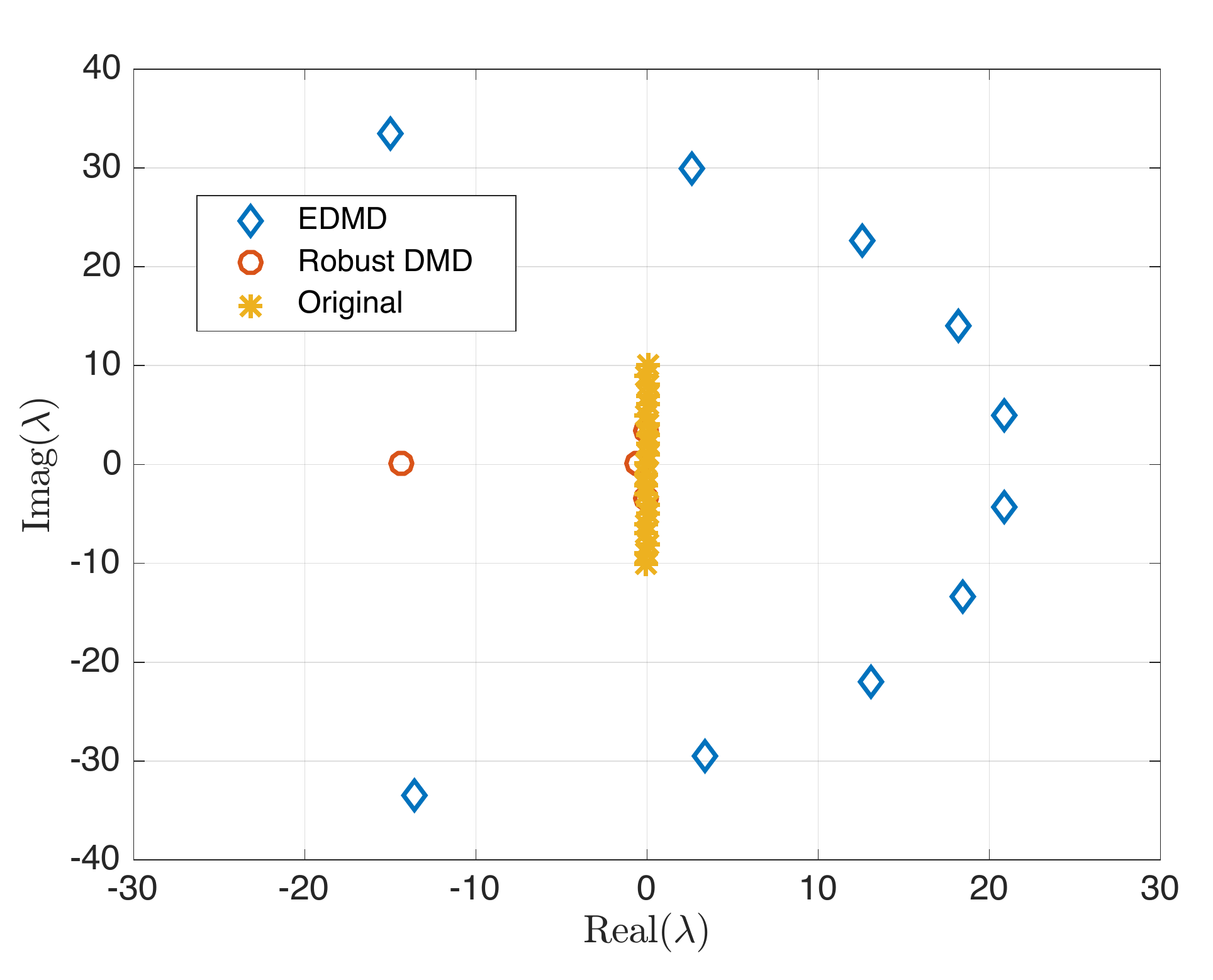}}
\caption{(a) Eigenvalues obtained with EDMD on original data and Robust EDMD on enriched data set. (b) Dominant eigenvalues.}\label{eig_nonlinear}
\end{figure}

The first 30-time steps data were used as the training data for training the Koopman operator. An extra 30 artificial points were added to the obtained data set to form the enriched data set and this enriched data set was used to compute the eigenspectrum of the Koopman operator using Robust EDMD algorithm. The eigenvalues obtained using the dictionary functions given in (\ref{stuart_dic}), with normal EDMD and Robust EDMD with enriched data set is shown in Fig. \ref{eig_nonlinear}a. Fig. \ref{eig_nonlinear}b shows the dominant eigenvalues and it can be observed that Robust EDMD provides a better approximation of the original eigenspectrum. In particular, normal EDMD generates unstable eigenvalues.

\begin{figure}[htp!]
\centering
\includegraphics[scale=.45]{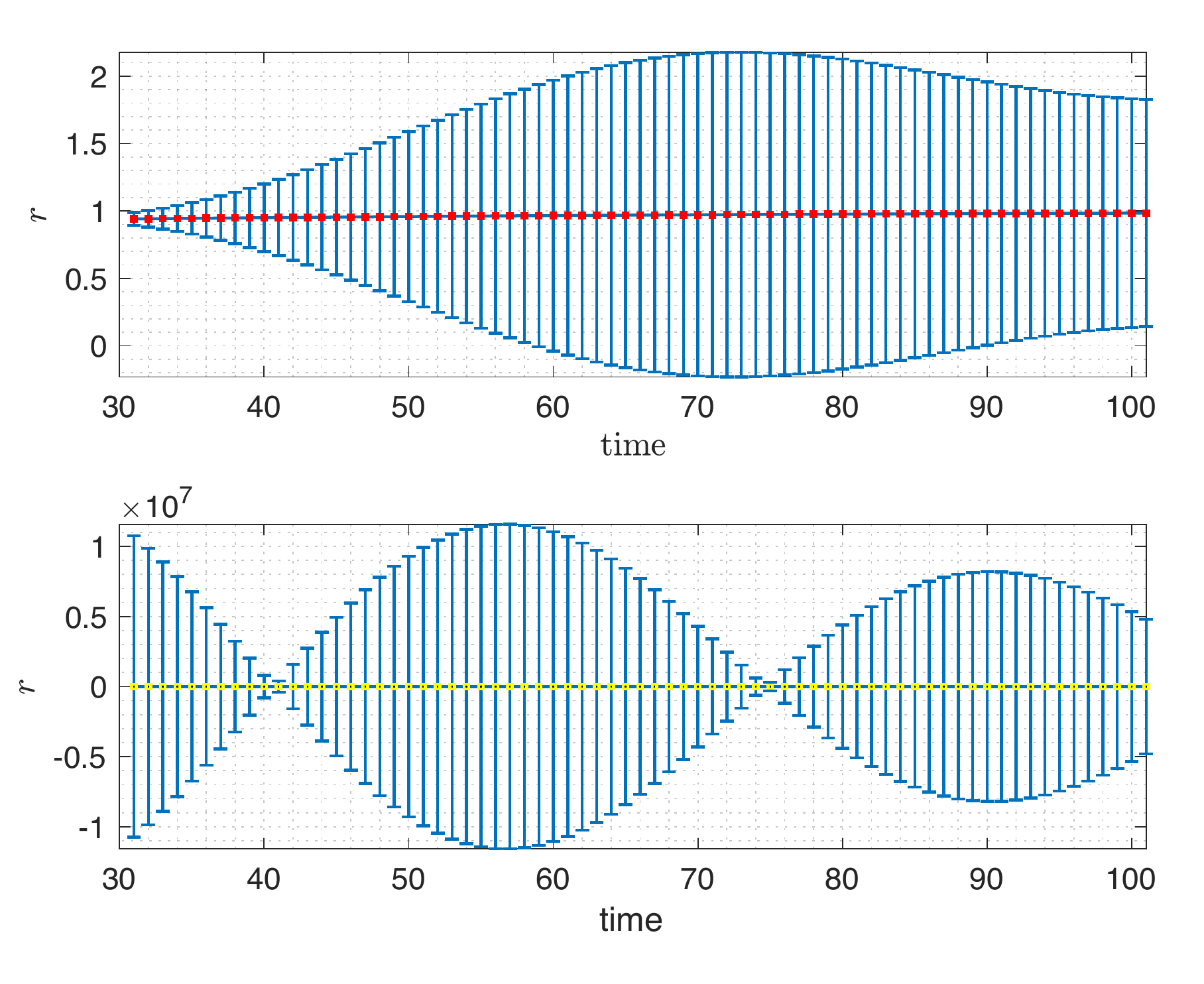}
\caption{Comparison of errors in prediction of $r$ using Robust EDMD and normal EDMD. The top figure shows the prediction error using Robust EDMD and the lower plot shows prediction error using normal EDMD.}\label{prediction_r}
\end{figure}

\begin{figure}[htp!]
\centering
\includegraphics[scale=.45]{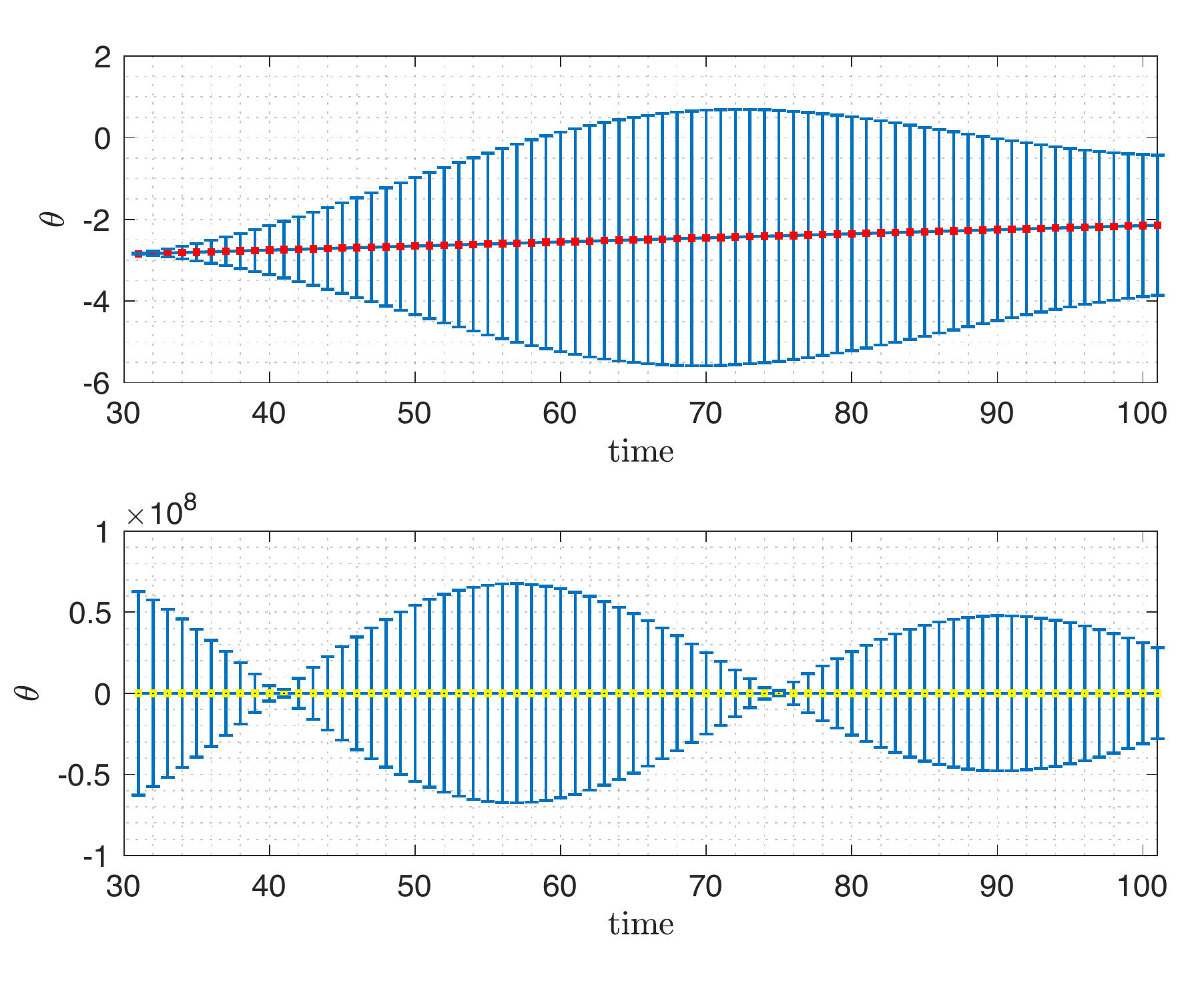}
\caption{Comparison of errors in prediction of $\theta$ using Robust EDMD and normal EDMD. The top figure shows the prediction error using Robust EDMD and the lower plot shows prediction error using normal EDMD.}\label{prediction_theta}
\end{figure}

Further, using the Koopman operators obtained using both normal EDMD and Robust EDMD, future values of both $r$ and $\theta$ was predicted for the next 70 time steps. The errors in the prediction of $r$ and $\theta$ are shown in Fig. \ref{prediction_r} and Fig. \ref{prediction_theta} respectively. In all the error plots, the errors are plotted against the actual values of $r$ and $\theta$ and it can be observed that the errors in prediction for both $r$ and $\theta$ with Robust EDMD are significantly smaller than the prediction errors using normal EDMD.

\subsection{Burger-Equation}
The third example considered in this paper is the Burger equation. Burger equation is a successful but simplified partial differential equation which describes the motion of viscous compressible fluids. The equation is of the form 
\[\partial_t u(x,t)+u\partial_x u=k\partial_x^2u\]
where $u$ is the speed of the gas, $k$ is the kinematic viscosity, $x$ is the spatial coordinate and $t$ is time. 

In the simulation, choosing $k=0.01$, we approximated the PDE solution using the Finite Difference method \cite{KUTLUAY1999251} with the initial condition $u(x,0)=sin(2\pi x)$ and Dirichet boundary condition $u(0,t)=u(1,t)=0$. Given the spatial and temporal ranges, $x\in[0,1],\;t\in[0,1]$, the discretizaion steps are chosen as $\Delta t=0.02$ and $\Delta x=1\times10^{-2}$. With the above set of conditions, the flow $u$ is shown in Fig. \ref{burger_flow}.
\begin{figure}[htp!]
\centering
\includegraphics[scale=.4]{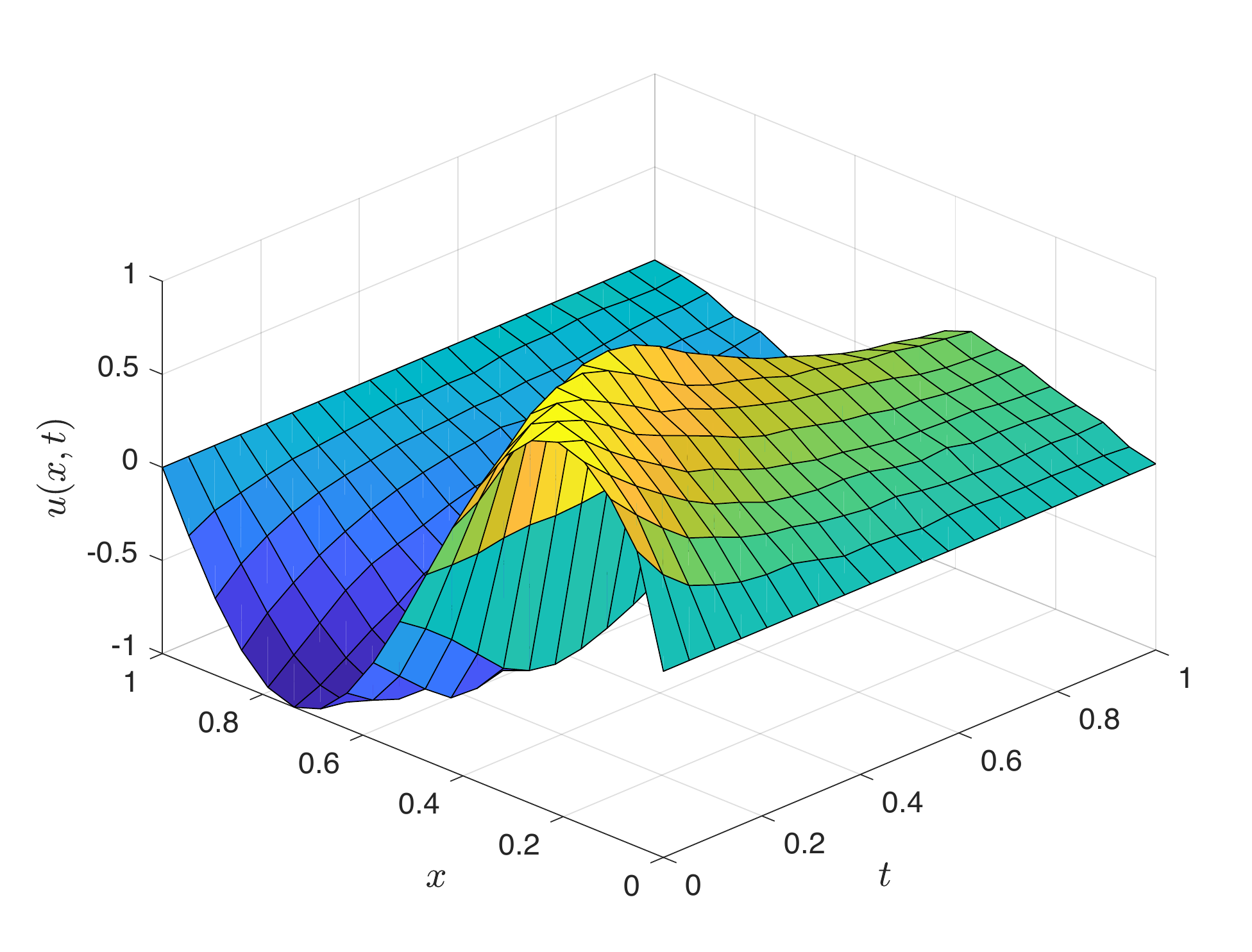}
\caption{Flow field of Burger equation.}\label{burger_flow}
\end{figure}

Since the space discretization was chosen as $\Delta x = 1\times 10^{-2}$, there are 100 state variables. For computing the Koopman operator, 8-time steps data were used. 40 extra data points were added to enrich the data set and the Robust Koopman operator was computed using the enriched data set. Koopman operator using normal DMD was also computed for comparing the errors in prediction. The errors in the prediction of 35 future time steps for $x_{40}$ and $x_{100}$ is shown in Fig. \ref{burger_error}(a) and Fig.. \ref{burger_error}(b) respectively. It can be seen that the error in prediction using Robust Koopman operator from the enriched data set is much smaller as compared to the normal DMD.

\begin{figure}[htp!]
\centering
\subfigure[]{\includegraphics[scale=.355]{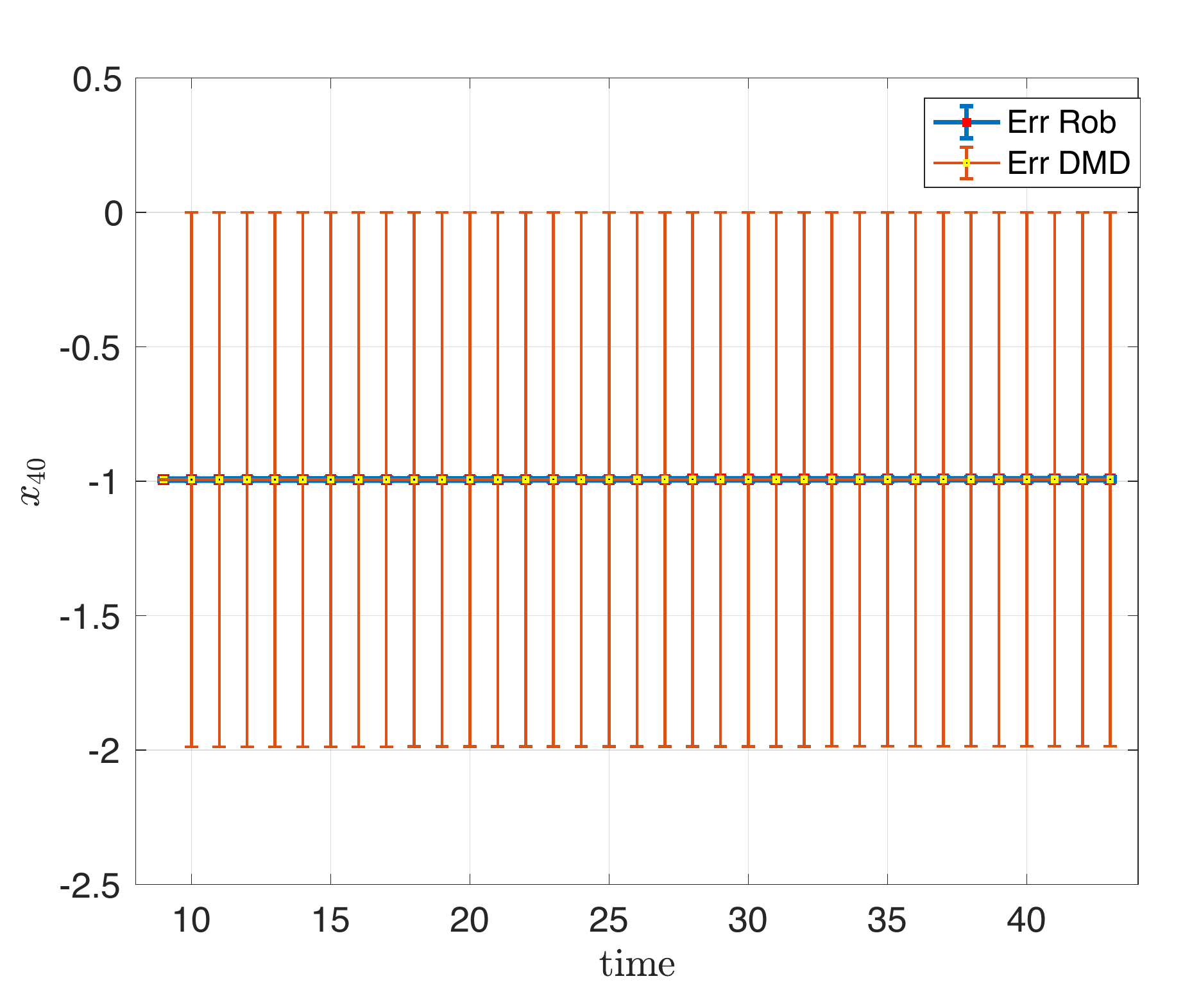}}
\subfigure[]{\includegraphics[scale=.35]{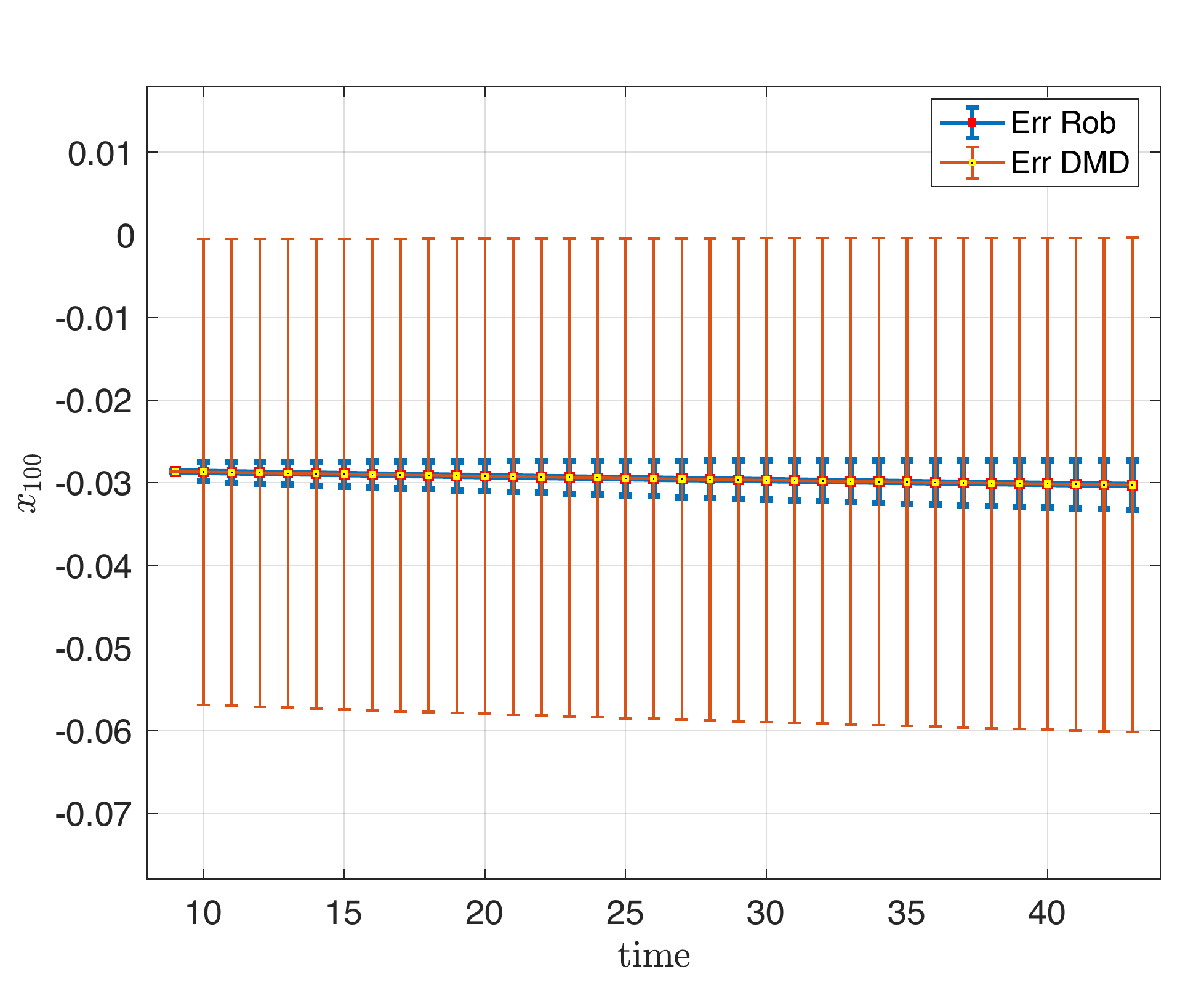}}
\caption{(a) Errors in prediction of $x_{40}$. (b) Errors in prediction of $x_{100}$.}\label{burger_error}
\end{figure}

We further used different training size data for computing the Koopman operator and compared the mean square error in prediction of all the states. In particular, we used both Robust DMD approach and normal DMD to predict 35-time steps from $t=100$, with 7 different training size data, namely 5, 10, 15, 20, 25, 30 and 35-time steps. For each of the training size data, we appended the data set with artificial data points so that there are 40 data points in total. The mean square errors in the prediction of the states are shown in Fig. \ref{burger_mse}. 

\begin{figure}[htp!]
\centering
\subfigure[]{\includegraphics[scale=.45]{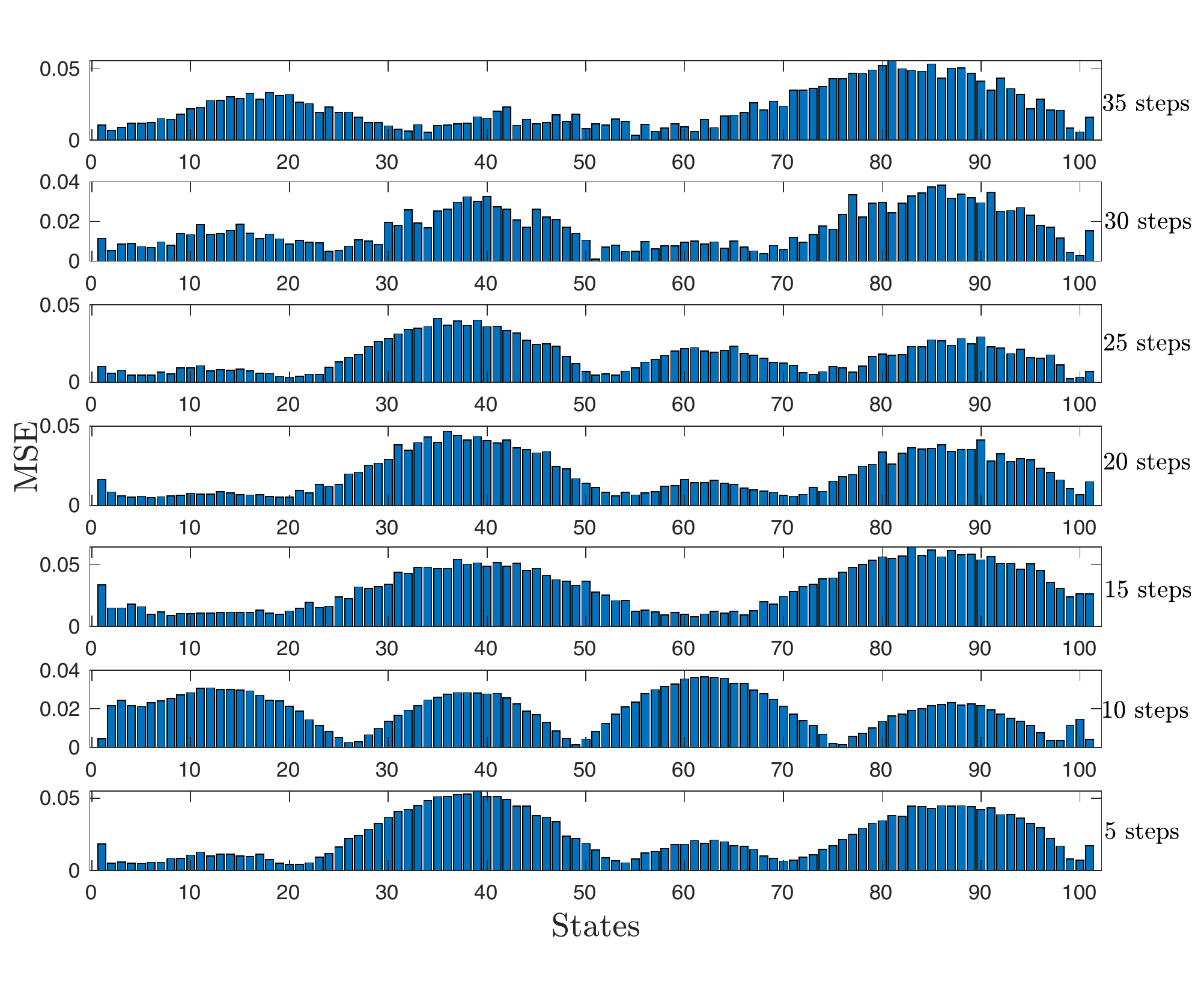}}
\subfigure[]{\includegraphics[scale=.45]{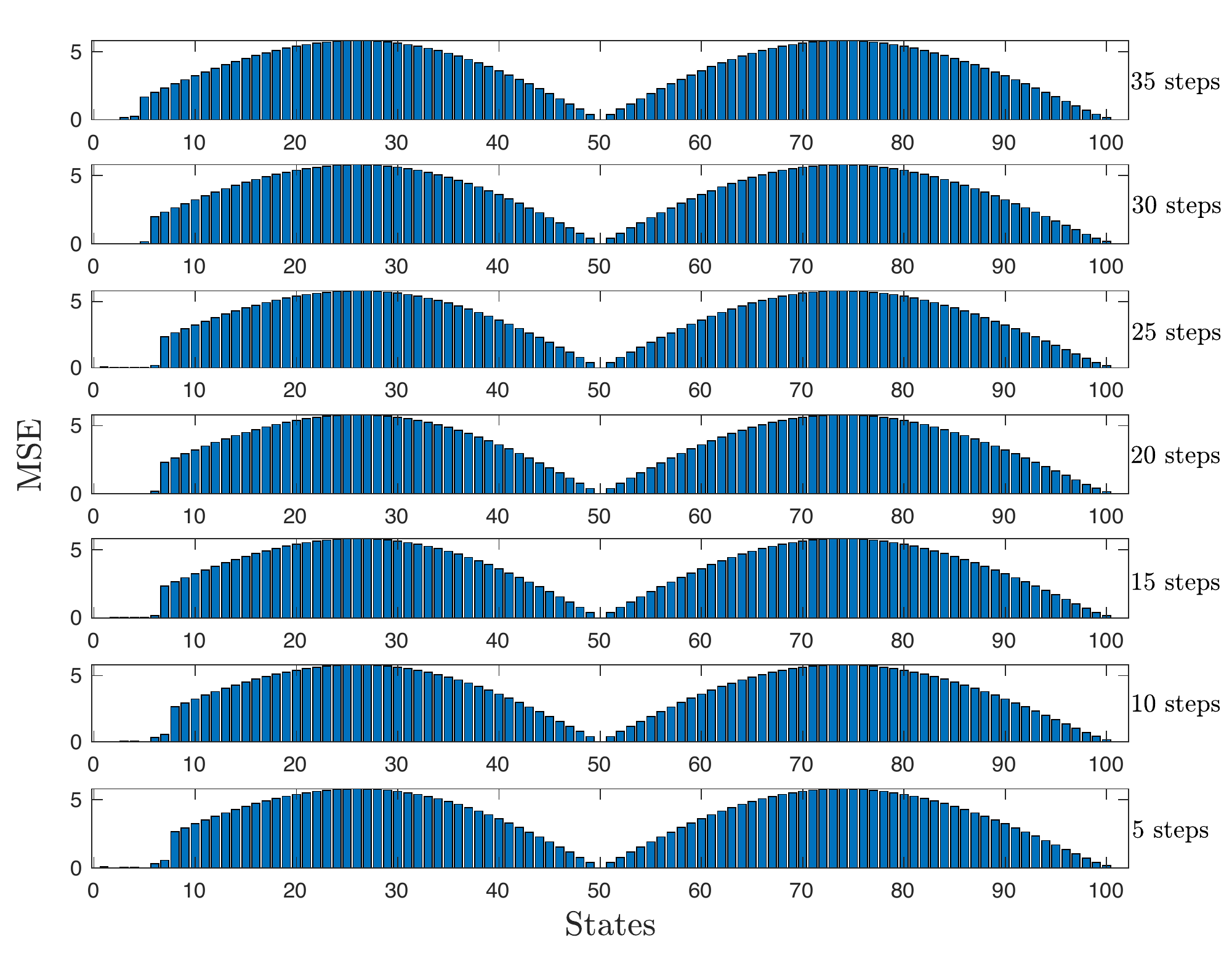}}
\caption{(a) Mean square error of prediction of all the states using Robust DMD approach. (b) Mean square error of prediction of all the states using normal DMD.}\label{burger_mse}
\end{figure}

Fig. \ref{burger_mse}(a) shows the mean square error in prediction using the proposed approach and Fig. \ref{burger_mse}(b) shows the mean square error using normal DMD. It can be clearly seen that errors using the proposed method are much smaller (of the order of $10^2$). Another observation is that normal DMD is not much sensitive to small variations in training data size, whereas the proposed method is more sensitive to training data size.

\section{Conclusions}\label{section_conclusion}

In this paper, we addressed the problem of computation of Koopman operator from sparse time series data. In certain experimental applications, it may not be possible to obtain time series data which is rich enough to approximate the Koopman operator. We propose an algorithm to compute the Koopman operator for such sparse data. The intuition was based on exploiting the differentiability of the system mapping to append artificial data points to the sparse data set and using robust optimization-based techniques to approximate the Koopman eigenspectrum. The efficiency of the proposed method was also demonstrated on three different dynamical systems and the results obtained were compared to existing Dynamic Mode Decomposition and Extended Dynamic Mode Decomposition algorithms to establish the advantage of our proposed algorithm and in the future; we hope to investigate the performance of our approach on real experimental data sets.

\bibliographystyle{IEEEtran}
\bibliography{subhrajit_robust_DMD}

\end{document}